\documentclass{amsart}

\usepackage{amsmath,amsfonts,amssymb}
\usepackage{xcolor,mhsetup}
\usepackage[extra,safe]{tipa}
\usepackage{mathtools}
\usepackage[left=2cm,right=2cm,top=2cm,bottom=2cm]{geometry}
\usepackage[applemac]{inputenc}
\usepackage{mathrsfs}
\usepackage{mathabx}
\usepackage{bm}
\usepackage{stmaryrd}
\usepackage{euscript}
 \usepackage{pstricks}
\usepackage{amsthm}
\usepackage{latexsym,amsmath,amscd}
\usepackage{graphicx}
\usepackage{color}
 \usepackage{dsfont}
\usepackage{enumerate}
\usepackage[normalem]{ulem}
\usepackage{pifont}
\usepackage{aliascnt}
\usepackage[hidelinks]{hyperref}
\usepackage{setspace}

\theoremstyle{definition}
\newtheorem{theorem}{Theorem}[section]
\newtheorem{prop}{Proposition}[section]

\newtheorem{lem}{Lemma}[section]
\newtheorem{corol}{Corollary}[section]

\theoremstyle{remark}
\newtheorem{remark}[theorem]{Remark}

\numberwithin{equation}{section}
\newcommand{\brem }{\begin{rem} \rm }
\newcommand{\erem }{\end{rem}}
\newcommand{\brems }{\begin{rems} \rm }
\newcommand{\erems }{\end{rems}}
\newcommand{\bhyp }{\begin{hyp} \rm }
\newcommand{\ehyp }{\end{hyp}}
\newcommand{\bex}{\begin{ex} \rm }
\newcommand{\eex}{\end{ex}}

\def\my_c{c_\infty}

\def \z{{\mathbf{z}}}

\newcommand{\mynewtheorem}[2]{
  \newaliascnt{#1}{dummy}
  \newtheorem{#1}[#1]{#2}
  \aliascntresetthe{#1}
  \expandafter\def\csname #1autorefname\endcsname{#2}
}

\newcommand{\be}{\begin{equation}}
\newcommand{\ee}{\end{equation}}
\newcommand{\bde}{\begin{displaymath}}
\newcommand{\ede}{\end{displaymath}}
\newcommand{\beq}{\begin{eqnarray*}}
\newcommand{\eeq}{\end{eqnarray*}}
\newcommand{\beqa}{\begin{eqnarray}}
\newcommand{\eeqa}{\end{eqnarray}}
\newcommand{\bel }{\left\{\begin{array}{ll}}
\newcommand{\eel}{\cr \end{array} \right.}

\newcommand{\seq}[1]{{\lbrace #1 \rbrace}}

\newcommand{\dcb}{\begin{array}{lll}}
\newcommand{\dce}{\end{array}}
\newcommand{\ebe}{\begin{enumerate}\setlength{\baselineskip}{13pt}\setlength{\parskip}{0pt}}
\newcommand{\dbe}{\end{enumerate}}

\def\s{\bold{s}}

\def\z{\bold{z}}

\newcommand{\E}{\mathcal{E}}

\def\F{{\cal F}}

\def\rr{{\mathbb R}}

\def\Q{{\mathbb Q}}
\def\P{{\mathbb P}}

\def\I{\mathsf{1}}

\newcommand \A[1]{{\bf (#1)}}

\def\F{{\mathcal F}}

\def\R{{\mathbb{R}} }
\def\N{{\mathbb{N}} }
\def\E{{\mathbb{E}}  }

\def\P{{\mathbb{P}}  }
\def\Q{{\mathbb{Q}}  }

\def\I{{\mathbf{1}}}

\def\bint#1^#2{\displaystyle{\int_{#1}^{#2}}}
\def\bsum#1^#2{\displaystyle{\sum_{#1}^{#2}}}

\def\xdt_#1{X_#1(\Delta t)}

\def\0{{\mathbf{0}}}

\begin{document}

\title{Weak uniqueness and density estimates for SDEs with coefficients depending on some path-functionals}

\author{Noufel Frikha}
\address{Noufel Frikha, Laboratoire de Probabilités et Modèles Aléatoires, UMR 7599, Universit\'e Paris Diderot, Paris VII,  
Bâtiment Sophie Germain, 5 rue Thomas Mann, 75205 Paris CEDEX 13}
\email{frikha@math.univ-paris-diderot.fr}



\author{Libo Li}
\address{Libo Li, School of Mathematics and Statistics, University of New South Wales, Sydney, Australia}
\email{libo.li@unsw.edu.au }


\date{\today}


\begin{abstract}
In this paper, we develop a general methodology to prove weak uniqueness for stochastic differential equations with coefficients depending on some path-functionals of the process. As an extension of the technique developed by Bass \& Perkins \cite{bass:perk:09} in the standard diffusion case, the proposed methodology allows one to deal with processes whose probability laws are singular with respect to the Lebesgue measure. To illustrate our methodology, we prove weak existence and uniqueness in two examples : a diffusion process with coefficients depending on its running symmetric local time and a diffusion process with coefficients depending on its running maximum. In each example, we also prove the existence of the associated transition density and establish some Gaussian upper-estimates.

\end{abstract}

\maketitle

\section{Introduction}
 
In the present paper, we investigate the weak existence and uniqueness of a one-dimensional stochastic differential equation (SDE in short) with coefficients depending on some path-functional $A$ and dynamics given by
\begin{equation}
\label{sde:dynamics}
X_t = x + \int_0^{t} b(X_{s},A_s(X)) ds + \int_{0}^t \sigma(X_s,A_s(X)) dW_s, \ t\in [0,T]
\end{equation}

\noindent where $(W_t)_{t \geq 0}$ stands for a one-dimensional Brownian motion and $(A_t(X))_{t \geq 0}$ is an $\rr^{d-1}$-valued functional depending on the path $X$, $d\geq2$. Some examples include its local and occupation times, its running maximum or minimum, its first hitting time of a level, its running average, etc. From the point of view of applications, systems of the type \eqref{sde:dynamics} appear in many fields. Let us mention stochastic Hamiltonian systems where $A_t(X)=\int_0^t F(s,X_s) ds$, see e.g. \cite{Soize} for a general overview, \cite{Talay} for convergence to equilibrium or \cite{Barucci} for an application to the pricing of Asian options. We also mention \cite{Forde20112802}, where the author constructs a weak solution to the SDE \eqref{sde:dynamics} with $b\equiv0$ and $A_t(X)=\max_{0\leq s \leq t}X_s$ is the running maximum of $X$ and investigates an application in mathematical finance. 

In the standard multi-dimensional diffusion framework, the martingale approach initiated by Stroock and Varadhan turns out to be particularly powerful when trying to get uniqueness results. It is now well-known that the martingale problem associated to a multi-dimensional diffusion operator is well posed as soon as the drift is bounded measurable and that the diffusion matrix is continuous (with respect to the space variable) and strictly positive, see e.g. Stroock and Varadhan \cite{stro:vara:79}. In the indicated framework, uniqueness is derived from Calder\'{o}n-Zygmund estimates. Also, when $a$ is H\"older continuous, an analytical approach using Schauder estimates can be applied, see e.g. Friedman \cite{frie:64}.

Recently, Bass and Perkins \cite{bass:perk:09} introduced a new technique for proving uniqueness for the martingale problem and illustrated it in the framework of non-degenerate, non-divergence and time-homogeneous diffusion operators under the assumption that the diffusion matrix is strictly positive and H\"older continuous.  It has also been recently extended by Menozzi \cite{Menozzi} for a class of multi-dimensional degenerate Kolmogorov equations that is the case of a multi-dimensional path functional $A=(A^1_t,\cdots, A^{N}_t)_{t\geq0}$ given by: $A^1_t(X) = \int_0^t F_1(X_s, A_s(X)) ds$, $A^2_t(X) = \int_0^t F_2(A^{1}_s, \cdots, A^{N}_s) ds$, $\cdots$, $A^{N}_t(X) = \int_0^t F_N(A^{N-1}_s,A^{N}_s) ds$, under an assumption of weak H\"ormander type on the functions $(F_1,\cdots, F_N)$. The approach in the two mentioned papers consists in using a perturbation method for Markov semigroups, known as the parametrix technique, such as exposed in Friedman \cite{frie:64} in the case of uniformly elliptic diffusion. More precisely, the first step of the strategy is to approximate the original system by a simple process obtained by freezing the drift and the diffusion coefficients in the original dynamics, and use the fact that the transition density of such approximation as well as its derivatives can be explicitly estimated. Then, the key ingredient is the \emph{smoothing property} of the underlying \emph{parametrix kernel}, see assumption \A{H1} (iv) in Section \ref{perturb:section} for a precise statement. This property reflects the quality of the approximation of the original dynamics. An important remark is that this \emph{smoothing property} is only achieved when the freezing point, that is the point where the coefficients are evaluated in the approximation process, is chosen to be the terminal point where the transition density is evaluated.

The main purpose of this paper is to develop a technique in order to prove weak uniqueness as well as existence of a transition density for some SDEs with path-functional coefficients where the probability law of the couple $(X_t,A_t(X))$ may be singular with respect to the Lebesgue measure on $\rr^d$. The main new feature added here compared to previous works on this topic is that our technique enables us to deal with a process whose probability law is absolutely continuous with respect to a $\sigma$-finite measure.

Our methodology can be summed up as follows: suppose that the transition density of the Markov process $(W_t, A_t(W))_{t\geq0}$ with initial point $x$  exists with respect to a $\sigma$-finite measure $\nu(x, dy)$, not necessarily being the Lebesgue measure on $\rr^d$, and a chain rule (It\^o's) formula for $f(t, X_t,A_t(X))$ is available, where $f$ belongs to a suitable class of functions $\mathcal{D}$ related to the domain of the infinitesimal generator of the Markov process $(W_t, A_t(W))_{t\geq0}$. Then as soon as the derivatives of the transition density of $(W_t, A_t(W))_{t\geq0}$ satisfies some good estimates or equivalently if the \emph{parametrix kernel} enjoys a smoothing property with respect to $\nu(x,dy)$, one has the main tools to prove weak uniqueness for the SDE \eqref{sde:dynamics}.

Since the probability law of the process $(X_t,A_t(X))_{t\geq0}$ may be singular, it is not clear how to select the approximation process and even if this crucial \emph{smoothing property} will be achieved in such context. Let us be more precise on one example. If one considers the couple $(X_t, A_t(X))_{t\geq0}$, $A_t(X)=L^{0}_t(X)$ being the symmetric local time at point $0$ accumulated by $X$ up to time $t$, it is easy to see that on $\left\{ T_0 > t \right\}$, $T_0$ being the first hitting time of $0$ by $X$, one has $L^{0}_t(X)=0$ whereas on $\left\{T_0 \leq t \right\}$, the process may accumulate local time so that the probability law of the couple $(X_t, L^{0}_t(X))$ consists in two parts, one being singular with an atom in the local time part, the other one (hopefully) being absolutely continuous with respect to the Lebesgue measure on $\rr \times \rr_+$. Hence, we see that for such dynamics the situation is more challenging than in the standard diffusion setting. This new difficulty will be overcome by choosing two independent parametrix kernels, one for each part, and then by proceeding with a non-trivial selection of the freezing point according to the singular measure induced by the approximation process. Even in this singular framework, one is able to prove that the \emph{smoothing property} of the parametrix kernel still holds which is as previously mentioned the keystone to prove weak uniqueness for the SDE \eqref{sde:dynamics}. As far as we know these feature appears to be new.

It will be apparent in what follows that our approach is not limited to the one-dimensional SDE case and can be easily adapted to multi-dimensional examples such as the one studied in \cite{Menozzi}. However, we decide to confine our presentation to the one-dimensional framework to foster understanding of the main arguments. As an illustration of our methodology we consider two examples. In the first one, we choose $A_t(X)=L^{0}_t(X)$, $t\geq0$, where $L^{0}_t(X)$ is the symmetric local time at $0$ accumulated by $X$ up to time $t$. The bivariate density of Brownian motion and its running symmetric local time at $0$ can be found e.g. in Karatzas \& Shreve \cite{kararzas:shreeve:local}. In the second example, we consider the running maximum of the process, that is $A_t(X) = \max_{0\leq s \leq t} X_s$, $t\geq0$. 
 

The first part of our main results can be seen as an application of the general methodology developed in Section \ref{weak:uniqueness:framework}. More precisely, in Theorem \ref{weak:existence:local:time} and Theorem \ref{bivariate:density:maximum} we prove that weak uniqueness holds for the SDE \eqref{sde:dynamics} when the path-functional is the symmetric local time or the running maximum, under the assumption that the drift $b$ is bounded measurable and the diffusion coefficient $a=\sigma^2$ is uniformly elliptic, bounded and $\eta$-H\"older-continuous, for some $\eta \in (0,1]$. 

Finally, the strategy developed in this paper can be used not only to prove the existence, but also to retrieve an explicit representation of the transition density (with respect to a $\sigma$-finite measure) of $(X_t, A_t(X))_{t \geq 0}$ as an infinite series. Obviously, such results are out of reach by using standard Malliavin calculus techniques, which cannot be employed here under such rather mild smoothness assumptions on the coefficients, or the Fourier transform approach developed in Fournier and Printemps \cite{fournier:printems:10}. 

However, one has to overcome new technical issues compared to the standard diffusion setting investigated by Friedman \cite{frie:64} or even to the degenerate case considered by Delarue and Menozzi \cite{dela:meno:10}. Leaving this technical discussion to Section \ref{uniqueness:density:local:time}, we only point out that the main difficulty lies in the non-integrable time singularity induced by the mixing of the singular and non-singular parts of the parametrix kernel. In order to overcome this issue, which to our best knowledge appears to be new, the key idea is to use the symmetry of the density of the killed proxy with respect to the initial and terminal points, in order to retrieve the integrability in time of the underlying parametrix kernel. As the second part of our main results, we prove the existence of the transition density for $(X_t, A_t(X))_{t \geq 0}$ as well as its representation in infinite series for the two examples mentioned before, see Theorem \ref{pn:local:time} and Theorem \ref{pn:max} below. Some Gaussian upper-estimates are also established.

\subsection*{Notations:}\rule[-10pt]{0pt}{10pt}\\
\label{sec:not}We introduce here some basic notations and definitions used throughout this paper. For a sequence of linear operators $(S_i)_{1\leq i \leq n}$, we define $\prod_{i=1}^{n} S_i = S_1 \cdots S_n$. We will often use the convention $\prod_{\emptyset} = 1$ which appears when we have for example $ \prod_{i=0}^{-1} $. Let $\mathcal{J}$ be a subset of $\rr^d$, we denote by $\mathcal{C}^{k}_b(\mathcal{J})$, the collection of bounded continuous functions which are $k$-times continuously differentiable with bounded derivatives in the interior of $\mathcal{J}$. The derivatives at the boundary $\partial \mathcal{J}$ are defined as limits from the interior and it is assumed that they exist and are finite. The set $\mathcal{B}_b(\mathcal{J})$ is the collection of real-valued bounded measurable maps defined on $\mathcal{J}$. Furthermore we will use the following notation for time and space variables $\s_p=(s_1,\cdots,s_p)$, $\z_p=(z_1,\cdots,z_p)$, the differentials $d\s_p = ds_1 \cdots ds_p$, $d\z_p = dz_1 \cdots dz_p$ and for a fixed time $t \geq 0$, we denote by $\Delta_p(t)  = \left\{ \s_p\in [0,t]^p: s_{p+1}:=0 \leq s_p\leq s_{p-1} \leq \cdots \leq s_1 \leq t=:s_0 \right\}$. For a multi-index $\alpha=(\alpha_1, \cdots, \alpha_\ell)$ of length $\ell$, we sometimes write $\partial_\alpha f(x) = \partial_{x_{\alpha_1}}\cdots \partial_{x_{\alpha_\ell}}f(x)$, for a vector $ x $. 
We denote by $y \mapsto g(ct,y)$ the transition density function of the standard Brownian motion with variance $c$, i.e. $g(ct,y) = (2\pi t c)^{-1/2}\exp(-y^2/(2tc))$, $y\in \R$. The associated non-normalized Hermite polynomials are defined respectively as $H_i(ct,y)= \partial^{i}_y g(ct,y)$, $i\in \N$.  For a fixed given point $z\in \rr^d$, the Dirac measure is denoted by $\delta_z(dx)$. For $a,b\in \rr$, we use the notation $a \asymp b$ if there exists a constant $C>1$ such that $C^{-1} b \leq a \leq C b$. We denote by $|f|_\infty$ the supremum norm of a function $f$. Throughout the paper, we will often use the space-time inequality $\forall x\in \rr$, $|x|^p e^{-q x^2} \leq (p/(2qe))^{p/2}$, valid for any $p,q>0$ and sometimes will omit to refer to it explicitly.\\

\section{Abstract framework for weak uniqueness}\label{weak:uniqueness:framework}


\subsection{A perturbation formula}\label{perturb:section}\rule[-10pt]{0pt}{10pt}\\
Throughout this section, we assume that there exists a weak solution $(X,W)$, $\left\{\mathcal{F}_t \right\}$ to \eqref{sde:dynamics} and that the process $Y_t:=(X_t,A_t(X))_{t\geq0}$, starting from the initial point $x$ at time $0$, lives on a closed space $\mathcal{J} \subset \rr^d$. The process $(X_t,A_t(X))_{t\geq0}$ induces a probability measure $\P^x$ (or simply denoted $\P$) on $\Omega=\mathcal{C}([0,\infty), \mathcal{J})$ which is endowed with the canonical filtration $(\mathcal{F}_t)_{t\geq0}$. We consider the collection of linear maps $(P_t)_{t\geq0}$ defined by $P_t f(x):=\E[f(X_t,A_t(X))]$ for $f\in \mathcal{B}_b(\mathcal{J})$. It is important to point out that, at this moment, we do not know whether $Y$ is a (strong) Markov process or not. However, one of our main assumption (see assumption \A{H1} (ii) below) links the process $Y$ to the solution of a martingale problem in the sense of Stroock \& Varadhan \cite{stro:vara:79}. As in the standard diffusion case, the strong Markov property will be a consequence of weak uniqueness. 

We define the shift operators $(\theta_ty)(s)= y(t+s)$, $0\leq s < \infty$ for $t \geq 0$ and $y \in \mathcal{C}([0,\infty), \mathcal{J})$. For any deterministic time $t\geq0$, we denote by $\mathbb{Q}_{t,w}$ the regular conditional probability for $\P^x$ given $\mathcal{F}_t$ and let $\P_{t,w}$ be the probability measure given by $\P_{t,w}:= \Q_{t,w} \circ \theta^{-1}_t$. In particular, for every $F \in \mathcal{B}(\mathcal{C}([0,\infty),\mathcal{J}))$, one has $\P^x(\theta^{-1}_t F\,|\,\mathcal{F}_t)(w) = \mathbb{Q}_{t,w}(\theta^{-1}_tF) = \P_{t,w}(F)$, for $\P^x$-a.e. $w\in \Omega$. 

As a substitute for the Markov property, we will assume that there exists a weak solution starting from $x=y(t)$ at time $0$ such that the probability measure induced by this solution is exactly $\P_{t,w}$, for every $w\notin N$, $N$ being a $\P^x$-null event of $\mathcal{F}_t$. As in the standard diffusion case, this will be a standard consequence of existence of solutions to the associated martingale problem. This assumption allows to prove that the finite-dimensional distributions are unique once we know that two weak solutions have the same one-dimensional marginal distributions, see e.g. Chapter 5 in \cite{kara:shre:91} or Chapter VI in \cite{bass:97}. 

We denote the approximation or the proxy process by $\bar X$, that is the solution of \eqref{sde:dynamics} with $b = 0$ and the diffusion coefficient $\sigma$ evaluated at some fixed point $z \in \mathcal{J}$. Without going into details at this point, the key idea is to consider the process $Y$ as a perturbation of the proxy $\bar Y = (\bar X_t,A(\bar X_t))_{t\geq0}$ whose law is denoted by $\bar p_t(x,dy) = \bar p_t^z(x, dy)$. Accordingly, we define the collection of linear maps $(\bar P_t)_{t\geq 0}$ by $\bar{P}_t f(x) :=\mathbb{E}[f(\bar{X}_t, A_t(\bar{X}))] = \int f(y) \bar{p}_t(x,dy)$ for $f\in \mathcal{B}_b(\mathcal{J})$. To indicate that one is working with the approximation or the proxy process with coefficients frozen at $z$ or functions associated with the approximation process frozen at $z$, we put a bar on top of the function. To indicate that the frozen point is the terminal point $y$ of the proxy density, we will put a hat instead of a bar. To derive a first order expansion of $(P_t)_{t\geq0}$ and prove weak uniqueness, we make the following assumptions:\\

\noindent {\bf Assumptions (H1):} Given the initial and frozen point $x,z\in \mathcal{J}$.
\begin{enumerate}
\item[(i)] 
\begin{enumerate}
\item[(a)] The proxy process $\bar Y^z$ is a Markov process with infinitesimal generator $\mathcal{\bar L}^z$.
\item[(b)] There exists a $\sigma$-finite measure $\nu(x,.)$ such that for all $t>0$, the law of $\bar Y^z_t$ is absolutely continuous with respect to $\nu(x,.)$. More specifically, there exists a $\nu(x,dy)$-integrable function $(t,x,y)\mapsto \bar p^z_t(x, y)$ satisfying
\bde
\bar p^z_t(x, dy) = \bar p^z_t(x, y) \nu(x,dy) \label{hypi}
\ede 
\noindent and $\bar P^z_t f(x)=\int f(y) \bar p^z_t(x, y) \nu(x,dy)$ for all $f \in \mathcal{B}_b(\mathcal{J})$.
\end{enumerate}
\item[(ii)]  There exists a class of functions $\mathcal{D}\subset \mathrm{Dom}( \mathcal{\bar{L}}^z)\cap\mathcal{C}_b(\mathcal{J})$ and a linear operator $\mathcal{L}$ acting on $\mathcal{D}$ such that:
\begin{enumerate}
\item[(a)] For all $g\in \mathcal{C}^{\infty}_b(\mathcal{J})$, $\bar{P}^z_t g \in \mathcal{D}$.
\item[(b)] For all functions $h$ such that: $s\mapsto h(s,.)\in \mathcal{C}^{1}(\rr_+,\mathcal{D})$, the process 
\begin{gather*}
h(t, Y_t) - h(0,x) - \int_0^t \left\{ \partial_1 h(s, Y_s) + \mathcal{L}h(s, Y_s) \right\} ds, \quad t\geq 0
\end{gather*}
is a continuous square integrable martingale under $\P^x$.
\item[(c)] There exists a \emph{parametrix kernel} $\bar{\theta}_t$ with respect to the measure $\nu$, that is a measurable map $(t, z, x, y) \mapsto \bar \theta^{z}_t(x, y)$ such that for all $g\in \mathcal{C}^{\infty}_b(\mathcal{J})$
\begin{align}
(\mathcal{L}- \mathcal{ \bar L}^z)\bar P^z_tg (x) = \int g(y)  \bar \theta^{z}_t(x, y) \nu(x,dy),  \, \quad t>0.\label{hypii}
\end{align}
\end{enumerate}
\item[(iii)]  For all $x, y \in \mathcal{J}$, the maps $(t,z)\mapsto\bar p_t^{z}(x, y)$ and $(t,z) \mapsto \bar \theta^{z}_{t}(x, y)$ are continuous on $(0,\infty) \times \mathcal{J}$.\\
\item[(iv)]  For all $t>0$, there exists some $\nu(x,dy)$-integrable functions $p^*_t(x,y)$, $\theta^{*}_t(x,y)$, a constant $\bar \zeta = \bar \zeta(z)\in \rr$ and a positive constant $C$, eventually depending on $t$ but in a non-decreasing way, such that
\bde
 |\bar p_t^{z}(x, y)| \leq p^*_t(x, y), \quad |\bar \theta_t^{z}(x, y)| \leq \theta^*_t(x, y), \quad \mathrm{and}  \quad \int | \theta^{*}_t(x, y)| \nu(x,dy) \leq C t^{\bar \zeta}.
\ede 
For the case $z = y$, we assume that the parametrix kernel enjoys the following \emph{smoothing property:} there exists $\zeta>-1$ and a positive constant $C$, eventually depending on $t$ in a non-decreasing way, such that 
\be
\forall t>0,\, \forall x\in \mathcal{J}, \quad\int  |\bar{\theta}^{y}_t(x,y)| \nu(x,dy) \leq C t^\zeta. \label{hypiv}
\ee 

\item[(v)] For any $g\in \mathcal{C}_b(\mathcal{J})$, one has
\begin{align*}
\lim_{t\downarrow 0}\int g(y)\bar p_t^y(x, y)\nu(x, dy)= g(x).
\end{align*}
\end{enumerate}

\noindent For notational simplicity, we define for $t>0$,
\begin{align*}
\hat{p}_t(x, y) 	&:= \bar{p}^y_t(x, y), \\
\hat P_t f(x)  		& := \int f(y) \hat p_t(x, y) \nu(x,dy) = \int f(y) \bar p^{y}_t(x, y) \nu(x,dy),\\
\mathcal{S}_{t}g(x) & := \int g(y)\hat \theta_{t}(x,y)  \nu(x,dy) = \int g(y)\bar\theta^y_{t}(x,y)  \nu(x,dy).
\end{align*}

\begin{remark} The set of assumptions \A{H1} will allow us to prove a perturbation formula of the map $P_t$ around $\hat{P}_t$, see Theorem \ref{t1} below. The kernel of $\mathcal{S}_t$ defined above satisfies the smoothing property (iv) which is, as mentioned in the introduction, the key point to prove uniqueness in law for equation \eqref{sde:dynamics}. This smoothing property was exploited in \cite{bass:perk:09} and then in \cite{Menozzi} for some degenerate Kolmogorov equations. The main new feature added here is that we are able to deal with a process that admits a density with respect to a $\sigma$-finite measure (with eventually several atoms). In particular, the process can be singular in the sense that it may not admit a transition density with respect to the Lebesgue measure on $\mathcal{J}$. 

Assumption \A{H1} (ii) b) provides a chain rule formula for the process $Y=(X_t, A_t(X))_{t\geq0}$ for a suitable class of functions $\mathcal{D}$ included in the domain of $\bar{\mathcal{L}}$. The operator $\mathcal{L}$ is identified by means of this chain rule formula. As we will see in Section \ref{examples}, this assumption will help us to formulate the martingale problem associated to the process $Y$. This will be used later on in order to establish the existence of a weak solution to the SDE \eqref{sde:dynamics}. It is important to point out that we don't know if $Y$ is a (strong) Markov process for the moment. In general, as in the standard diffusion case, this will be a consequence of weak uniqueness, or equivalently of the well-posedness of the martingale problem, see \cite{stro:vara:79}, \cite{kara:shre:91} or \cite{bass:97}.
    
\end{remark}

\begin{theorem}\label{t1}
Assume that \A{H1} holds. Then, for any $g\in \mathcal{C}_b(\mathcal{J})$,
\begin{align*}
P_Tg(x) = \hat P_Tg(x) + \int^T_0 P_s \mathcal{S}_{T-s}g(x)\,ds.
\end{align*}
\end{theorem}
\begin{proof}
Let $f\in \mathcal{C}^{\infty}_b(\mathcal{J})$. For $t\in [0,T]$ and $r>0$, by assumption \A{H1} (ii) a) and b) applied to $(t,x) \mapsto h(t, x)= \bar{P}_{T-t+r} f(x)\in \mathcal{C}^1([0,T],\mathcal{D})$, there exists a continuous martingale $(\bar{M}_t)_{0\leq t  < T}$ starting at $0$ such that
\begin{align*}
\bar P_{T-t+r}f(Y_t) & = \bar P_{T+r}f(x) + \int^t_0 (\partial_s + \mathcal{L})\bar P_{T-s+r}f(Y_s)\, ds + \bar M_t\\
   				   & = \int f(y) \bar p_{T+r}^z(x, y) \nu(x,dy) + \int^t_0 \int f(y)\bar \theta^{z}_{T-s+r}(Y_s, y)\nu(Y_s,dy)\, ds + \bar M_t.
\end{align*}

We now proceed to the {\it diagonalisation argument}, that is the argument that allows one to select the freezing point $z$ according to the measure $\nu(x, dy)$. We consider a sequence of non-negative mollifiers $\delta^{z}_\varepsilon, \, \varepsilon>0$, such that $\delta^{z}_\varepsilon \leq C_\varepsilon$ and $(\delta^{z}_\varepsilon)_{\varepsilon>0}$ converges weakly to the Dirac mass at $z$ as $\varepsilon \rightarrow 0$. For $g\in\mathcal{C}^{\infty}_b(\mathcal{J})$, we apply the above decomposition for $f = \delta^{z}_\varepsilon g$ and take expectations. We obtain
$$
P_t \bar P_{T-t+r}\delta^{z}_\varepsilon g(x) = \int \delta^{z}_\varepsilon(y) g(y) \bar p^z_{T+r}(x, y) \nu(x, dy) + \int^t_0 \int \delta^{z}_\varepsilon(y) g(y) \mathbb{E}[ \bar \theta^{z}_{T-s+r}(Y_s, y)\nu(Y_s,dy)]\, ds.
$$

\noindent We let $t\rightarrow T$ and integrate with respect to $z$, by continuity, we get
\begin{equation}
\label{eq:before:limit}
\int P_T \bar P_{r}\delta^{z}_\varepsilon g(x) dz = \int \int \delta^{z}_\varepsilon(y) g(y) \bar p^z_{T+r}(x, y) \nu(x,dy) dz + \int^T_0 \int \mathbb{E}[ \int \delta^{z}_\varepsilon(y) g(y) \bar \theta^{z}_{T-s+r}(Y_s, y)\nu(Y_s,dy)\,]dz \, ds
\end{equation}

We pass to the limit as $\varepsilon \rightarrow 0$ and then let $r\rightarrow 0$ in \eqref{eq:before:limit}.
 We first give some useful estimates. By using  \A{H1} (iv) and \eqref{hypiv}, we have
\begin{align}
|\int  \delta^{z}_\varepsilon(y) \bar p^{z}_{r}(Y_T,y)dz|&  \leq  p^*_r(Y_T,y), \label{t3e1}\\
|(\mathcal{L} - \mathcal{\bar L}) \bar  P_{t} \delta^{z}_\varepsilon g(x)| & \leq C_\varepsilon |g|_\infty t^{\bar \zeta}. \label{t3e2}
\end{align}
Let us consider the left-hand side of \eqref{eq:before:limit}. From \eqref{t3e1}, we can apply Fubini's theorem to obtain
\begin{align}
\int \mathbb{E}[\bar P_{r} \delta^{z}_\varepsilon g(Y_T)] dz 
& = \mathbb{E}[ \int  \int  \delta^{z}_\varepsilon(y) g(y) \bar p^{z}_{r}(Y_T,y)\nu(Y_T,dy) dz] \label{t2e1}\\
& = \mathbb{E}[ \int \left\{\int  \delta^{z}_\varepsilon(y) \bar p^{z}_{r}(Y_T,y)dz \right\}g(y) \nu(Y_T,dy)] \nonumber.
\end{align}

By a similar argument, one can apply dominated convergence theorem together with \A{H1} (iii) to obtain 
\bde
\lim_{\varepsilon \downarrow 0} \int \mathbb{E}[\bar P_{r}\delta^{z}_\varepsilon g(Y_T)] dz = \mathbb{E}[ \int \bar p^{y}_{r}(Y_T,y)g(y) \nu(Y_T,dy)] = \mathbb{E}[\hat P_{r} g(Y_T)].
\ede
Consequently, by letting $r\rightarrow 0$ and using \A{H1} (iii) and (v), we obtain again by the dominated convergence theorem $\lim_{r\downarrow 0}\mathbb{E}[\hat P_{r} g(Y_T)] = P_Tg(x)$. 

For the right-hand side of \eqref{eq:before:limit}. Again by \eqref{t3e1} we can apply Fubini's theorem and then the dominated convergence theorem while having in mind \A{H1} (iii). This yields
\begin{align}
\lim_{\varepsilon \downarrow 0} \int \bar P_{T+r} \delta^{z}_\varepsilon g(x)dz & = \lim_{\varepsilon \downarrow 0}\int \left\{\int \delta^{z}_\varepsilon(y) g(y) \bar p^{z}_{T+r}(x,y) dz \right\}\nu(x,dy) \label{t2e}\\
& = \int g(y) \bar p^y_{T+r}(x,y) \nu(x,dy) = \hat P_{T+r}g(x).\nonumber
\end{align}
Letting $r\rightarrow 0$, by \A{H1} (iii) we deduce from the continuity of $r\mapsto \hat{P}_r g(x)$ that $\lim_{r\downarrow 0} \hat{P}_{T+r} g(x) = \hat{P}_T g(x)$.
The second term on the right-hand side of \eqref{eq:before:limit} is computed similarly by using \A{H1} (ii). To pass to the limit as $\varepsilon\rightarrow 0$, we apply the dominated convergence theorem using \A{H1} (iv) and (iii) we obtain
\begin{align}
& \lim_{\varepsilon \downarrow 0} \int^T_0 \int \mathbb{E}[ \int \delta^{z}_\varepsilon(y) g(y)  \bar \theta^{z}_{T-s+r}(Y_s, y)\nu(Y_s,dy)]dz \, ds \label{t2e2}\\ 
&= \lim_{\varepsilon \downarrow 0} \mathbb{E}[\int^T_0 \int \left\{\int g(y)\delta^{z}_\varepsilon(y)\bar \theta^{z}_{T-s+r}(Y_s, y)  dz\right\} \nu(Y_s,dy)\, ds ] \nonumber \\
&= \mathbb{E}[\int^T_0 \int\lim_{\varepsilon \downarrow 0} \left\{\int g(y)\delta^{z}_\varepsilon(y)\bar \theta^{z}_{T-s+r}(Y_s,y)  dz\right\} \nu(Y_s,dy)\, ds]\nonumber \\
& = \mathbb{E}[\int^T_0  \int g(y)\bar \theta_{T-s+r}^{y}(Y_s,y)  \nu(Y_s,dy)\, ds]  = \mathbb{E}[\int^T_0  \mathcal{S}_{T-s+r} g(Y_s) ds ] \nonumber 
\end{align}

\noindent and to pass to the limit as $r\rightarrow 0$, we remark that \A{H1} (iv) ($\zeta>-1$) implies the continuity at $0$ of $r \mapsto \mathbb{E}[\int^T_0  \mathcal{S}_{T-s+r} g(Y_s) ds]$. The result is valid for $g\in \mathcal{C}^{\infty}_b(\rr^d)$ and an approximation argument completes the proof.
\end{proof}

We are not so far from obtaining a representation of $P_T g$ in infinite series. Once weak existence and uniqueness for the SDE \eqref{sde:dynamics} is established, this representation will be useful in order to derive the existence of a transition density for the process $Y$. We will also use it to derive some Gaussian upper-bound estimates for the density of the couple $(X_T, A_T(X))$. 
\vskip13pt   

\begin{corol}\label{corolllary:semigroup}
Assume that \A{H1} holds and that for any $g\in \mathcal{C}_b(\mathcal{J})$ and $t>0$, the function $x\mapsto \mathcal{S}_t g$ defined in Section \ref{perturb:section} belongs to $\mathcal{C}_b(\mathcal{J})$. Then one may iterate the first order formula in Theorem \ref{t1} to obtain 
\begin{align}\label{param:series}
P_T g(x) = \hat{P}_T g(x) + \sum_{n\geq1} I^n_T g(x)
\end{align}
\noindent with 
$$
I^{n}_T g(x) = \int_{\Delta_n(T)} d\s_n \hat{P}_{s_n} \mathcal{S}_{s_{n-1}-s_{n}} \cdots \mathcal{S}_{T-s_1}g(x).
$$
Moreover, the series \eqref{param:series} converges absolutely and uniformly for $x\in \mathcal{J}$.
\end{corol}

\begin{proof}
We remark that since for all $t>0$, $x \mapsto \mathcal{S}_{t}g(x) \in \mathcal{C}_b(\mathcal{J})$, we can iterate the first order expansion in Theorem \ref{t1} to obtain
$$
P_Tg(x) = \hat{P}_T g(x) + \sum_{n=1}^{N-1} \int_{\Delta_n(T)} d\s_n \hat{P}_{s_n} \mathcal{S}_{s_{n-1}-s_n}\cdots \mathcal{S}_{T-s_1}g(x) + \mathscr{R}^{N}_{T}g(x)
$$

\noindent where the remainder term is given by
$$
\mathscr{R}^{N}_Tg(x) = \int_{\Delta_N(T)} d\s_N P_{s_N} \mathcal{S}_{s_{N-1}-s_N}\cdots \mathcal{S}_{T-s_1}g(x).
$$
From iterative application of estimate \eqref{hypiv}, the remainder term is bounded by
\begin{align*}
|\mathscr{R}^N_T g(x)| & \leq C_T |g|_{\infty} \int_{\Delta_N(T)} d\s_N \prod_{n=0}^{N-1} C (s_{n}-s_{n-1})^{\zeta} \leq C^{N}_T T^{N(1+\zeta)} \frac{\Gamma(1+\zeta)^{N}}{\Gamma(1+ N(1+\zeta))}
\end{align*}

\noindent where $\zeta \mapsto \Gamma(\zeta)$ is the Gamma function. 

By Lemma \ref{beta:type:integral} and the asymptotics of the Gamma function at infinity, we clearly see that the remainder goes to zero uniformly in $x\in \mathcal{J}$ as $N\uparrow \infty$.
\end{proof}

\vskip13pt
\subsection{Weak Uniqueness} \rule[-10pt]{0pt}{10pt}\\
Throughout this section, we will assume that \A{H1} holds and prove weak uniqueness for the SDE \eqref{sde:dynamics}. The main argument is an extension of the technique introduced by Bass \& Perkins \cite{bass:perk:09}, which allows us to deal with singular probability law in the sense that the law of $Y_t$, $t>0$ may not be absolutely continuous with respect to the Lebesgue measure on $\rr^d$. Moreover, the new contribution of this section compared to the existing literature on this topic is that we identify the main assumptions \A{H1} and \A{H2} needed to establish weak uniqueness, thus allowing for a general treatment.    

We consider two weak solutions of the SDE \eqref{sde:dynamics} starting at time $0$ from the same initial point $x\in \mathcal{J}$. Denote by $\P_1$ and $\P_2$ the two probability measures induced on the space $(\mathcal{C}([0,\infty),\mathcal{J}), \, \mathcal{B}(\mathcal{C}([0,\infty),\mathcal{J})))$. Define $f\in \mathcal{B}_b(\mathcal{J})$ and $\lambda >0$ 
\begin{align*}
S^i_\lambda f(x) & := \mathbb{E}_i [\int^\infty_0  e^{- \lambda t} f(Y_t) dt] = \int^\infty_0 e^{- \lambda t} \mathbb{E}_i [f(Y_t)]\,dt, i=1,2,\, \quad S^\Delta _\lambda f(x)  := (S^1_\lambda - S^2_\lambda) f(x)\\
\|S^\Delta_\lambda\| & := \sup_{\|f\|_\infty \leq 1} |\mathcal{S}^{\Delta}_\lambda f|.
\end{align*}

 We notice that by \A{H1} (ii) b),
\be\label{res}
S^{i}_\lambda (\lambda - 	\mathcal{L}) f(x) = f(x), \quad \forall f\in \mathcal{D}, \, i=1,2.
\ee 
For $z\in \mathcal{J}$, the resolvent of the process with frozen coefficients is defined by
\bde
 \bar{R}_\lambda f(x) = \int^\infty_0 e^{-\lambda t} \bar{P}_tf(x) dt, \quad \quad \forall f\in \mathcal{B}_b(\mathcal{J})
\ede 
\noindent and for $f\in \mathcal{D}$ one has 
\be\label{res:proxy}
 \bar{R}_\lambda (\lambda- \bar{\mathcal{L}})f = f.
\ee
We make the following assumptions:\\

\noindent {\bf Assumption (H2)}: For all $\lambda>0$, one has $\bar R_\lambda \mathcal{D} \subset \mathcal{D} $ and for $f\in \mathcal{D}$, 
\bde
(\lambda - \mathcal{\bar  L})\bar R_\lambda f= \bar{R}_{\lambda} (\lambda-\bar{\mathcal{L}})f \quad  and \quad  (\mathcal{L}- \mathcal{\bar L})  \bar R_\lambda f(x)
 = \int^\infty_0 e^{-\lambda t}(\mathcal{L} - \mathcal{\bar L}) \bar  P_{t } f(x)  \,dt.
\ede

Let $z\in \rr^d$ and $r>0$. We consider a sequence of non-negative mollifiers $\delta^{z}_\varepsilon, \, \varepsilon>0$, converging to the Dirac mass at $z$ as $\varepsilon \rightarrow 0$. Let us first observe that if $g\in \mathcal{C}^{\infty}_b(\mathcal{J})$ then $\bar P_r g\in \mathcal{D}$, and by \A{H2} and \eqref{res:proxy}, one has
\begin{align}
(\lambda - \mathcal{L})\bar R_\lambda \bar P_r  \delta^z_\varepsilon g(x) &=
(\lambda - \mathcal{\bar  L})\bar R_\lambda \bar P_r  \delta^z_\varepsilon g(x) - (\mathcal{L}- \mathcal{\bar L})\bar R_\lambda \bar P_r  \delta^z_\varepsilon g(x)\nonumber\\ 
& =   \bar P_r \delta^z_\varepsilon g(x) - (\mathcal{L}- \mathcal{\bar L})  \bar R_\lambda \bar P_r  \delta^z_\varepsilon g(x).\label{LR1}
\end{align}
Note that the second term appearing in the right-hand side of the above equality can be expressed as
\be
(\mathcal{L}- \mathcal{\bar L})  \bar R_\lambda \bar P_r  \delta^z_\varepsilon g(x)
 = \int^\infty_0 e^{-\lambda t}(\mathcal{L} - \mathcal{\bar L}) \bar  P_{t + r}  \delta^z_\varepsilon g(x)  \,dt. \label{LR2}
\ee 
We are now ready to prove weak uniqueness for \eqref{sde:dynamics}. 
\begin{theorem}\label{weak:uniqueness:result}
Assume \A{H1} and \A{H2} are satisfied, then weak uniqueness holds for the SDE \eqref{sde:dynamics}.
\end{theorem}
\begin{proof}
The first part of the proof is similar to that of \autoref{t1}. We integrate both sides of \eqref{LR1} with respect to $dz$, apply $S^\Delta_\lambda$ and then pass the limit as $\varepsilon, r\rightarrow 0$.
\noindent For $i=1,2$, using estimates \eqref{t3e1} and \eqref{t3e2} we can apply Fubini's theorem. This yields
\begin{align*}
S_\lambda^i \int (\lambda - \mathcal{L})\bar R^{z}_\lambda \bar P^{z}_r  \delta^{z}_\varepsilon g \, dz &= \int^\infty_0 e^{-\lambda t} \mathbb{E}_i[\int (\lambda - \mathcal{L})\bar R^{z}_\lambda \bar P^{z}_r  \delta^{z}_\varepsilon g(Y_t) dz] dt= \int S_{\lambda}^i(\lambda - \mathcal{L})\bar R^{z}_\lambda \bar P^{z}_r  \delta^{z}_\varepsilon g dz.
\end{align*}

Then by using the fact that $\bar R_\lambda \bar P_r  \delta_\varepsilon g \in \mathcal{D}$ and \A{H2}, we deduce
\bde
S^\Delta_\lambda (\int  \bar P_r^{z} \delta_\varepsilon^{z} g \,dz) - S^\Delta_\lambda (\int (\mathcal{L}- \mathcal{\bar L}^{z})  \bar R_\lambda^{z} \bar P_r^{z}  \delta_\varepsilon^{z} g\, dz) = 0.
\ede

Let us consider the first term in the above expression and take the limit as $\varepsilon \rightarrow 0$ and then let $r \rightarrow 0$. For $i=1,2$, under \A{H1}, the limit in $\varepsilon$ can be taken using dominated convergence theorem while the limit as $r\rightarrow 0$ follows using \A{H1} (v) (similar to \eqref{t2e}),
\begin{align*}
\lim_{r\downarrow 0} \lim_{\varepsilon\downarrow 0}\int^\infty_0 e^{-\lambda t}\mathbb{E}_i[\int \bar P^{z}_r \delta^{z}_\varepsilon g(Y_t) dz ] dt 
& = \lim_{r\downarrow 0} \int^\infty_0 e^{-\lambda t}\mathbb{E}_i[\hat P_rg(Y_t)] dt = \int^\infty_0 e^{-\lambda t}\mathbb{E}_i[g(Y_t)] dt.
\end{align*}
This shows that $\lim_{r,\varepsilon \downarrow 0} S^\Delta_\lambda (\int  \bar P_s^{z}\delta_\varepsilon^{z} g \,dz)(x) = S^\Delta_\lambda g(x)$. For the second term, we first rewrite it using \eqref{LR2}. For $i=1,2$, since $r>0$, the Fubini's theorem can be applied due to \eqref{t3e2}. Next, from \A{H1} (v), the dominated convergence theorem can be applied to pass to the limit in $\varepsilon$,
\begin{align*}
& \lim_{\varepsilon\downarrow 0}   \int^\infty_0 e^{-\lambda u}\mathbb{E}_i[\int\int e^{-\lambda t}  (\mathcal{L} - \mathcal{\bar L}^{z}) \bar  P^{z}_{t +r} \delta^{z}_\varepsilon g(Y_u) \,dt dz ] du\\
& =   \int^\infty_0 e^{-\lambda u}\int^\infty_0 e^{-\lambda t}  \lim_{\varepsilon\downarrow 0}  \mathbb{E}_i[ \int (\mathcal{L} - \mathcal{\bar L}^{z}) \bar  P^{z}_{t + r} \delta^{z}_\varepsilon g(Y_u) dz] \,dt du\\
& =\int^\infty_0 e^{-\lambda u} \mathbb{E}_i[\int^\infty_0 e^{-\lambda t}  \mathcal{S}_{t+r} g(Y_u) dt] du.
\end{align*}

\noindent where the last equality follows from similar arguments as those employed in \eqref{t2e2}. One can now let $r$ goes to zero by using estimates in \A{H1} (iv) or \eqref{hypiv} to obtain
\begin{align*}
\lim_{r\downarrow 0}\int^\infty_0 e^{-\lambda u} \mathbb{E}_i[\int^\infty_0 e^{-\lambda t}  \mathcal{S}_{t+r} g(Y_u) dt\,] du  = \int^\infty_0 e^{-\lambda u} \mathbb{E}_i[\int^\infty_0 e^{-\lambda t}  \mathcal{S}_{t} g(Y_u) dt\,] du.
\end{align*}

By putting the two terms together, we obtain
\begin{align*}
S^\Delta_\lambda g = S^\Delta_\lambda (\int^\infty_0 e^{-\lambda t}  \mathcal{S}_{t} g \,dt).
\end{align*}

\noindent and one can pick $\lambda$ such that 
\bde
| \int^\infty_0 e^{-\lambda t}  \mathcal{S}_{t} g dt|  \leq  |g|_{\infty}  \int^\infty_0 e^{-\lambda t}  t^\zeta dt = |g|_{\infty} \frac{\Gamma(\zeta)}{\lambda^{1+\zeta}} < \frac{1}{2}|g|_{\infty},
\ede 

From the above computation and the definition of $\|S^{\Delta}_{\lambda}\|$, we find that 
\bde
|S^\Delta_\lambda g | = |S^\Delta_\lambda (\int^\infty_0 e^{-\lambda t}  \mathcal{S}_{t} g \,dt)| \leq \frac{1}{2} \|S^\Delta_\lambda\| |g|_\infty,
\ede

By an approximation argument, the last inequality remains valid for bounded continuous functions $g$ supported in $\mathcal{J}$ and, by a monotone class argument, it extends to bounded measurable functions. Taking the supremum over $|g|_\infty \leq 1$ yields $\|S^\Delta_\lambda\| \leq \frac{1}{2}\|S^\Delta_\lambda\|$ and, since $\|S^{\Delta}_\lambda \|< \infty$, we conclude that $S^\Delta_\lambda = 0$. Consequently, $\int_0^{\infty} e^{-\lambda t} \E_1[g(Y_t)] dt = \int_0^{\infty} e^{-\lambda t} \E_2[g(Y_t)] dt$. By the uniqueness of the Laplace transform together with continuity w.r.t the variable $t$, $ \E_1[g(Y_t)] =  \E_2[g(Y_t)]$ for all $t\geq0$ if $g$ is bounded measurable. 

Now one can use the standard argument based on regular conditional probabilities to show that the finite dimensional distributions of the process $(Y_t)_{t\geq 0} =(X_t,A_t(X))_{t\geq0}$ agree under $\P_1$ and $\P_2$. This is where we use the assumption on regular conditional probability measure introduced in the first paragraph of Section \ref{perturb:section}. Since this arguments is standard, we omit it. This suffices to prove weak uniqueness, see \cite{stro:vara:79}, Section 5.4.C in \cite{kara:shre:91} or Section VI.2 in \cite{bass:97}.
\end{proof}
\vskip10pt

\section{Two examples}\label{examples}\vskip5pt

\subsection{A diffusion process and its running symmetric local time} $\,$\vskip5pt
In this example, we consider the SDE with dynamics
\begin{equation}
\label{sde:dynamics:local:time}
X_t = x + \int_0^t b(X_s, L^{0}_s(X)) ds + \int_0^t \sigma(X_s,L^{0}_s(X)) dW_s
\end{equation}

\noindent where $L^{0}_s(X)$ is the symmetric local time at $0$ accumulated by $X$ at time $s$. Here $\mathcal{J} = \rr\times \rr_+$, $A_t(X)=L^{0}_t(X)$ and $d=2$. We introduce the following assumptions:
\bigskip
\begin{itemize}
\item[\A{R-$\eta$}] The coefficients $b$ and $a=\sigma^2$ are bounded measurable functions defined on $\rr \times \rr_+$. The diffusion coefficient $a$ is $\eta$-H\"older continuous on $\rr \times \rr_+$.  

\item[\A{UE}] There exists some constant $\underline{a}>0$ such that $\forall (x,\ell) \in \rr \times \rr_+$, $\underline{a} \leq a(x,\ell)$.

\end{itemize}

\bigskip 
Let $\mathcal{D}$ be the class of function $f \in \mathcal{C}^{2,1}_b(\rr\backslash{\left\{0\right\}} \times \rr_+) \cap \mathcal{C}_b(\rr\times \rr_+)$ such that $\partial_1f(0+,\ell)= \lim_{x\downarrow 0} \frac{f(x,\ell)-f(0,\ell)}{x}$ and $\partial_1 f(0-,\ell) = \lim_{x\uparrow 0} \frac{f(x,\ell) - f(0,\ell)}{x}$ exist, are finite and satisfy the following transmission condition:
\begin{equation}
\label{transmission:condition}
\forall \ell \in \rr_+, \quad \frac{\partial_1 f(0+,\ell) - \partial_1 f(0-,\ell)}{2} + \partial_2 f(0,\ell) = 0.
\end{equation}
We define the linear operator $\mathcal{L}$ by: 
\begin{align*}
\mathcal{L}f(x,\ell) & = b(x,\ell) \partial_1 f(x-,\ell) + \frac12 a(x,\ell) \partial^2_1 f(x-,\ell), \quad (x,\ell) \in \rr  \times \rr_+.
\end{align*}


As mentioned in the introduction, we need a chain rule formula for the process $(X_t,A_t(X))_{t\geq0}$ which allows to identify the set of functions $\mathcal{D}$ and the linear operator $\mathcal{L}$. In fact, the set $\mathcal{D}$ described above is precisely the set of functions for which we are able to provide a good characterisation of the martingale problem. Indeed, one has to rely on the following chain rule formula or generalisation of the It\^o formula whose proof closely follows the arguments of Theorem 2.2 in Elworthy \& al. \cite{Elworthy:Truman:Zhao} or Theorem 2.1 in Peskir \cite{Peskir}. Note that here that we are working with the symmetric local time at zero whereas the right local time is considered in \cite{Elworthy:Truman:Zhao}.

\begin{prop}[Generalised It\^o's formula] \label{generalized:ito:lemma}Assume that $f\in \mathcal{C}^{1,2,1}(\rr_+\times \rr\backslash \left\{0\right\} \times \rr_+) \cap \mathcal{C}(\rr_+\times \rr \times \rr_+)$ satisfies: $\partial_2 f(t,0+,\ell) = \lim_{x\downarrow 0} (f(t,x,\ell)-f(t,0,\ell))/x$ and $\partial_2 f(t,0-,\ell) = \lim_{x\uparrow 0} (f(t,x,\ell)-f(t,0,\ell))/x$ exist and are finite. Then, one has
\begin{align*}
f(t,X_t, L^{0}_t(X)) & = f(0,x,0) + \int_0^t \left\{\partial_1 f(s,X_s, L^{0}_s(X)) + \mathcal{L}f(s,.)(X_s,L^{0}_s(X)) \right\} ds \\
& \quad + \int_0^t \left\{\frac{\partial_2 f(s,0+,L^{0}_s(X))-\partial_2f(s,0-,L^{0}_s(X))}{2} + \partial_3 f(s,0,L^{0}_s(X)) \right\} dL^{0}_s(X)  \\
& \quad + \int_0^t \sigma(X_s,L^{0}_s(X)) \partial_2 f(s,X_s -,L^{0}_s(X)) dW_s \quad a.s.
\end{align*}
\end{prop}

\subsection{Weak Existence} $\,$\vskip5pt
Now that we have identified the set $\mathcal{D}$ and the linear operator $\mathcal{L}$, the weak existence of a solution to \eqref{sde:dynamics:local:time} follows from a standard compactness argument that we present here for sake of completeness. Actually, it is equivalent to the existence of a solution to the following martingale problem. 

We will say that a probability measure $P$ on $\left( \mathcal{C}([0,\infty),\rr\times \rr_+), \mathcal{B}(\mathcal{C}([0,\infty),\rr\times \rr_+ )\right)$ endowed with the canonical filtration $(\mathcal{F}_t)_{t\geq0}$ is a solution to the local martingale problem if $P(y_1(0)= x, y_2(0)=0) = 1$, $t\mapsto y_2(t)$ is a non-decreasing process $P$-$a.s.$ and
\begin{equation}
\label{martingale:characterisation}
f(y(t)) - f(y(0)) - \int_0^t \mathcal{L}f(y(s)) ds 
\end{equation}

\noindent is a continuous local martingale for $	f\in \mathcal{C}^{2,1}(\rr\backslash{\left\{0\right\}} \times \rr_+,\rr) \cap \mathcal{C}(\rr\times \rr_+)$ satisfying the transmission condition \eqref{transmission:condition}. We will say that $P$ is a solution to the martingale problem if \eqref{martingale:characterisation} is a continuous square integrable martingale for every $f\in \mathcal{D}$. Similarly to the standard diffusion case (see e.g. Proposition 5.4.11 in \cite{kara:shre:91}, since the coefficients $b$ and $\sigma$ are bounded, existence of a solution to the local martingale problem is equivalent to the existence of solution to the (non-local) martingale problem. 

We now claim that a probability measure that solves the local martingale problem induces a weak solution to the functional SDE \eqref{sde:dynamics:local:time}. Indeed, for the choices $(x,\ell) \mapsto x $ and $(x,\ell) \mapsto x^2$, following similar lines of proof of Proposition 5.4.6 in \cite{kara:shre:91}, one obtains that there exists a (one-dimensional) standard Brownian motion $W=\{ W_t, \tilde{\mathcal{F}}_t; 0\leq t < \infty \}$ eventually defined on an extension of the original probability space $(\mathcal{C}([0,\infty),\rr\times \rr_+), \mathcal{B}(\mathcal{C}([0,\infty),\rr\times \rr_+), P )$ such that
$$
y_1(t) = y_1(0)  + \int_0^t b(y(s)) ds + \int_0^t \sigma(y(s)) dW_s.
$$

We first consider $(x,\ell) \mapsto |x| - \ell \in \mathcal{C}^{2,1}( \rr\backslash \left\{0\right\} \times \rr_+) \cap \mathcal{C}(\rr \times \rr_+)$ and note that it satisfies the transmission condition \eqref{transmission:condition}. Hence, we derive that there exists a local martingale $M^1$ such that $|y_1(t)|-|x|- y_2(t) = \int_0^t b(y(s)) \mathrm{sign}(y_1(s)) ds + M^1_t$. Moreover, from the Tanaka formula $|y_1(t)|-|x| = \int_0^t \mathrm{sign}(y_1(s)) dy_1(s) + L_t^{0}(y_1)$. Hence, there exists a local martingale $M^2$ such that $L^{0}_t(y_1) - y_2(t) = M^2_t$. However, since $y_2$ is non-decreasing, this means that $M^2$ is a continuous local martingale of bounded variation. It follows that $M^2$ is identically equal to zero and $y_2(t)=L^{0}_t(y_1)$. Finally, we get that $\{y_1, W, \ (\tilde{\mathcal{F}}_t)_{t\geq 0}\}$ is a weak solution to the SDE \eqref{sde:dynamics:local:time}. Moreover, as in the standard diffusion case, see e.g. Chapter 6 in \cite{stro:vara:79} or Lemma 5.4.19 in \cite{kara:shre:91}, the measure $\P_{t, w}= \mathbb{Q}_{t, w} \circ \theta_t^{-1}$, where $\mathbb{Q}_{t,w}$ is a regular conditional probability for $P^x$ given $\mathcal{F}_t$, solves the martingale problem for every $w\notin N$, $N\in \mathcal{F}_t$ being a $P^x$-null event.  

We are now in position to prove the existence of a weak solution to the SDE \eqref{sde:dynamics:local:time}. The lines of reasoning here are standard, see e.g. Theorem 5.4.22 in \cite{kara:shre:91}. We provide them for sake of completeness.

\begin{theorem}\label{weak:existence:thm}Assume that the coefficients $b, \ \sigma: \rr \times \rr_+ \rightarrow \rr$ are bounded and continuous functions. Then, for every $x \in \rr$, there exists a weak solution to the SDE \eqref{sde:dynamics:local:time}. 
\end{theorem}

\begin{proof}
 Let us consider on some probability space $(\Omega, \mathcal{F}, \mathbb{P})$ a Brownian motion $W=\left\{W_t, \mathcal{F}^{W}_t, \ 0\leq t < \infty \right\}$ and let $\left\{ \mathcal{F}_t\right\}$ be the augmented filtration satisfying the usual conditions. For integers $j\geq0$, $n\geq1$ we consider the dyadic rationals $t^{(n)}_j = j 2^{-n}$, $j=0, \cdots, 2^{n}$ and introduce the functions $\psi_n(t) = t^{(n)}_j$, for $t \in [t^{(n)}_j, t^{(n)}_{j+1})$. For each integer $n\geq1$, we define the continuous process $y^{(n)}=\{(y^{(n)}_1(t),y^{(n)}_2(t)), \mathcal{F}_t; 0\leq t < \infty \}$ by setting $y^{(n)}(0)=(x,0)$ and then recursively for $t\in (t^{(n)}_{j}, t^{(n)}_{j+1}] $:
\begin{align*}
y^{(n)}_1(t) & = y^{(n)}_1(t^{(n)}_j) + b(y^{(n)}(t^{(n)}_j)) (t-t^{(n)}_j) + \sigma(y^{(n)}(t^{(n)}_j)) (W_t - W_{t^{(n)}_j}),   \\
y^{(n)}_2(t) & = L^{0}_t(y^{(n)}_1)
\end{align*}

\noindent for $j=0, \cdots, 2^{n}$. Defining the new coefficients for $y\in \mathcal{C}([0,\infty),\rr\times \rr_+)$
$$
b^{(n)}(t, y) := b(y_{\psi_n(t)}), \quad \sigma^{(n)}(t,y) := \sigma(y_{\psi_n(t)}), \quad t\geq0
$$

\noindent we remark that $y^{(n)}$ solves the functional SDE 
\begin{align*}
y^{(n)}_1(t)   & = x + \int_0^t b^{(n)}(s,y^{(n)}) ds + \int_0^t \sigma^{(n)}(s,y^{(n)}) dW_s \\
y^{(n)}_2(t) & = L^{0}_t(y^{(n)}_1)
\end{align*}

\noindent which clearly satisfies
$$
\sup_{n\geq1} \mathbb{E}[|y^{(n)}(t)-y^{(n)}(s)|^{2p}] \leq C(t-s)^{p}, \quad 0\leq s \leq t \leq T
$$

\noindent for some positive constant $C$ independent of $n$, $t$ and $s$. Thus, the sequence of probability measures $\mathbb P^{(n)} = \mathbb P \circ (y^{(n)})^{-1}$, $n\geq1$, is tight. Relabelling the indices if necessary, we may assert that $(\mathbb P^{(n)})_{n\geq1}$ converges weakly to a probability measure $\mathbb P^{*}$. 

Let $h\in \mathcal{D}$. We denote by $\left\{ \mathcal{B}_t \right\}$ the canonical filtration. It remains to prove that for every bounded, continuous function $f: \mathcal{C}([0,\infty),\rr\times \rr_+) \rightarrow \rr$ which are $\mathcal{B}_s$-measurable, one has
\begin{equation}
\mathbb{E}_{\mathbb P^*}\Big[\left\{ h(y(t))-h(y(s)) - \int_s^t \mathcal{L}h(y(v)) dv \right\} f(y)\Big] = 0.
\end{equation}

From Proposition \ref{generalized:ito:lemma}, one has
$$
\mathbb{E}_{\mathbb P^{(n)}}\Big[ \left\{h(y(t))-h(y(s)) - \int_s^t \mathcal{L}^{(n)}_v h(y^{(n)}) dv \right\} f(y)\Big] = 0,
$$

\noindent with
$$
\mathcal{L}^{(n)}_v h(y) = b^{(n)}(v, y) \partial_1 h(y(v)) + \frac12 a^{(n)}(v,y) \partial^2_1 h(y(v)) + \partial_3 h(y(v)) \I_\seq{y_1(v)\geq0}.
$$

Thus it remains to prove that $b^{(n)}$ and $a^{(n)}$ converges to $b$ and $a$ uniformly on compact subset of $\mathcal{C}([0,\infty),\rr\times \rr_+ \times \rr_+)$. Let $K $ be a compact subset of $\mathcal{C}([0,\infty),\rr\times \rr_+)$ so that
$$
M:= \sup_{w\in K, v \in [0,T]}|w(u)|< \infty, \quad \lim_{n\rightarrow +\infty} \max_{|t-s|\leq 2^{-n}, (s,t)\in [0,T]^2, w \in K} |w(t)-w(s)| =0
$$

 Since $b$ and $\sigma$ are uniformly continuous on $[-M,M] \times [0,M]$, for every positive $\varepsilon$ we can find an integer $n(\varepsilon)$ such that 
$$
\sup_{s \in [0,T], w\in K}\left\{ |b^{(n)}(s, w) - b(w)| + |\sigma^{(n)}(s, w) - \sigma(w)|\right\} \leq  \varepsilon, \quad n\geq n(\varepsilon)
$$

\noindent and the uniform convergence on $K$ of $(b^{(n)},a^{(n)})$ to $(b,a)$ follows. This completes the proof of Theorem \ref{weak:existence:thm}.
\end{proof}

\begin{remark}\label{weak:existence:bounded:mes:drift}We proved that weak existence for the SDE \eqref{sde:dynamics:local:time} holds under the assumption that $b$ and $\sigma$ are continuous and bounded. Using a transformation of the drift via the Girsanov theorem, one easily obtain weak existence under the assumption that $b$ is bounded measurable and $a=\sigma^2$ is continuous and uniformly elliptic. We omit technical details and refer the interested reader to Proposition 5.3.6 in \cite{kara:shre:91} for a similar argument in the standard diffusion setting. Hence, we conclude that under \A{R-$\eta$}, for some $\eta\in (0,1]$, and \A{UE}, weak existence holds for the SDE \eqref{sde:dynamics:local:time}. 
\end{remark}

\subsection{Weak uniqueness and representation of the transition density}\label{uniqueness:density:local:time}$\,$\vskip5pt
We now introduce the proxy process $\bar{X}_t:=x_0 + \sigma(z_1)W_t$, $t\geq0$, obtained from the original process $X$ by removing the drift part and by freezing the diffusion coefficient at $z_1=(x_1,\ell_1)\in \rr \times \rr_+$. For $f\in \mathcal{C}_b(\rr \times \rr_+)$, we define 
\begin{align*}
 \bar{P}_t f(x_0,\ell_0) & = \E[f(\bar{X}_t, \ell_0 + L^{0}_t(\bar{X}))] \\
 &  = \E[f(x_0 + \bar \sigma W_t, \ell_0 + L^{0}_t(\bar{X}))],
\end{align*}

\noindent where $\bar \sigma=\sigma(z_1)$, $\bar a = \bar \sigma^2$, and for $f\in \mathcal{D}$, the operator
\begin{align*}
\bar{\mathcal{L}}f(x,\ell) & =  \frac12 \bar{a} \, \partial^2_1 f(x-,\ell), \quad (x,\ell) \in \rr  \times \rr_+.
\end{align*}


We now compute the bivariate transition density of the approximation process $(\bar{X}_t, L^{0}_t(\bar{X}))_{t\geq0}$ from the joint density of $(W_t,L^{0}_t(W))$ which is readily available from Karatzas \& Shreve \cite{kararzas:shreeve:local}. 

We denote by $T_0=\inf\left\{ t\geq0: x_0 + \sigma(z_1) W_t = 0\right\}$ the first hitting time of $0$ by the process $(\bar{X}_t)_{t\geq0}$. Let $f\in \mathcal{C}_b(\rr \times \rr_+)$. We compute each term of the following decomposition:
\begin{align*}
\bar{P}_t f(x_0,\ell_0) & := \mathbb{E}[f(x_0+ \bar \sigma W_t, \ell_0) \I_\seq{T_0\geq t}] + \mathbb{E}[f(x_0+\bar  \sigma W_t, \ell_0+ L^{0}_t(\bar{X})) \I_\seq{T_0 < t}]\\
& =: I + II 
\end{align*}

In the first term, the Brownian motion does not accumulate local time at zero. The bivariate density of $(W_t, \max_{0\leq s \leq t}W_s)$, see e.g. \cite{kara:shre:91}, gives
$$
I = \int_{\rr \times \rr_+} f(x,\ell) \left\{ H_0(\bar{a} t, x-x_0) - H_0(\bar{a}t, x+x_0) \right\}  \I_\seq{x_0  x\geq0}  dx\delta_{\ell_0}(d\ell).
$$

To compute $II$ we make use of the bivariate density of $(W_t,L^{0}_t(W))_{t\geq0}$ established in \cite{kararzas:shreeve:local}. Conditioning with respect to $T_0$ and using the strong Markov property of $W$ yield
\begin{align*}
II & =  \int_0^t \P(T_0 \in ds) \mathbb{E}[f(x_0 + \bar \sigma W_t, \ell_0 + L^{0}_t(\bar{X}))| T_0=s] \\
 & =  \int_0^t  \P(T_0 \in ds) \mathbb{E}[f( \bar\sigma W_{t-s}, \ell_0 + \bar \sigma L^{0}_{t-s}(W))] \\
 & =  \int_0^t  ds\,  (-H_1)(s, \frac{|x_0|}{\bar \sigma}) \int_{\rr\times \rr_+} f(\bar \sigma x, \ell_0 + \bar \sigma \ell ) (-H_1)(t-s, |x| + \ell)  dx d\ell
\end{align*}

\noindent where we used the exact expression for the density of the passage time $T_0$, namely $\P(T_0\in ds) = (-H_1)(s,|x_0|/ \bar \sigma) ds$, $s>0$ and $x_0 \in \rr$. Since the sum of independent passage times is again a passage time, see e.g. page 824 of Karatzas \& Shreve \cite{kararzas:shreeve:local}, one has
\begin{equation}
\label{convolution:hitting:times}
(-H_1)(t, |x| + |y|) = \int_0^t ds \, (-H_1)(t-s,|x|) (-H_1)(s,|y|); \quad x,y \neq 0 , \quad t>0.
\end{equation}

Combining these observations with Fubini's theorem and a change of variable yield
\begin{align*}
II & = \int_{\rr\times \rr_+} f(x, \ell) \frac{1}{\bar a} (-H_1)(t, \frac{|x|+|x_0|+ \ell-\ell_0}{\bar \sigma}) \I_\seq{\ell_0 \leq \ell }dx d\ell .
\end{align*}

Combining $I$ and $II$, we see that the couple $(\bar{X}_t, \ell_0+ L^{0}_t(\bar X))$ admits a density $(x,\ell) \mapsto \bar{p}_t(x_0,\ell_0, x, \ell)$, that is
\begin{gather*}
\bar{p}_t(x_0,\ell_0, dx, d\ell)  = \bar{p}_t(x_0,\ell_0, x, \ell) \nu(x_0,\ell_0, dx, d\ell),  \quad t>0, 
\end{gather*}

\noindent with $\bar{p}_t(x_0,\ell_0, x, \ell) := \bar{f}_t(x_0,x) \I_\seq{\ell= \ell_0} + \bar{q}_t(x_0, \ell_0,  x, \ell)  \I_\seq{\ell_0 < \ell}$ and 
\begin{align*}
\bar{f}_t(x_0,x)&  := H_0(\bar a t, x-x_0) - H_0(\bar{a}t, x+x_0), \\
\bar{q}_t(x_0,\ell_0, x, \ell) &:= - \frac{1}{\bar{a}} H_1( t,  (|x|+|x_0| + \ell-\ell_0)/\bar \sigma), \\
\nu(x_0,\ell_0, dx, d\ell) &:=  \I_\seq{\ell_0 < \ell} dx d\ell + \I_\seq{x_0x \geq 0}  dx\delta_{\ell_0}(d\ell).
\end{align*}

Moreover, as already mentioned in assumption \A{H1} in Section \ref{perturb:section}, we let $\hat{p}_t(x_0,\ell_0, x,\ell) = \hat{f}_t(x_0,x) \I_\seq{\ell= \ell_0} + \hat{q}_t(x_0, \ell_0,  x, \ell)  \I_\seq{\ell_0 < \ell}$ with
\begin{align*}
\hat{f}_t(x_0,x)&  := H_0(a(x,\ell_0) t, x-x_0) - H_0(a(x,\ell_0)t, x+x_0), \\
\hat{q}_t(x_0,\ell_0, x, \ell) &:= - \frac{1}{a(x,\ell)} H_1( t,  (|x|+|x_0| + \ell-\ell_0)/\sigma(x,\ell)).
\end{align*}

 Hence, we see that both measures $\bar{p}_t$ and $\hat{p}_t$ consist of two parts, the first part is absolutely continuous with respect to the $\sigma$-finite measure $\I_\seq{x_0 x \geq 0}   dx  \delta_{\ell_0}(d\ell)$. Here the approximation process consists in freezing the diffusion coefficient at $(x,\ell_0)$. This is a natural idea since in this part the process $\bar{X}$ does not accumulate local time at zero and $\ell_0$ is the both the initial and terminal point of the density. The second part is absolutely continuous with respect to the Lebesgue measure on $\rr \times \rr_+$. Here the approximation process is obtained by freezing the diffusion coefficient at the terminal point of the density as in the standard diffusion case.
 
The first main result of this section establishes weak uniqueness for the SDE \eqref{sde:dynamics:local:time} by proving that assumptions \A{H1} and \A{H2} of Section \ref{perturb:section} are satisfied. Its proof is given in the Appendix. 

\begin{theorem} \label{weak:existence:local:time} For $\eta \in (0,1]$, under \A{R-$\eta$} and \A{UE}, weak uniqueness holds for the SDE \eqref{sde:dynamics:local:time}. 
\end{theorem}

In the following, we show that, given $(x_0, \ell_0) \in \rr \times \rr_+$, the transition density of the process $ (X^{x_0}_t, \ell_0 + L^0_t(X))_{t\geq0}$ is absolutely continuous with respect to the sigma finite measure $\nu(x_0, \ell_0, dx, d\ell)$. Our strategy consists in establishing a representation in infinite series of $P_t g$ from which stems an explicit representation of the density of the couple $(X^{x_0}_t, \ell_0 + L^0_t(X))$, see Theorem \ref{existence:bivariate:density:local:time}. Though we will not proceed in that direction, we point out that this representation may be useful in order to study the regularity properties of the density, to obtain integration by parts formulas or to derive an unbiased Monte Carlo simulation method. We refer the interested reader to \cite{frikha:li:kohatsu} for some results in that direction related to the first hitting times of one-dimensional elliptic diffusions. 

To this end, we need to iterate the first step expansion obtained in Theorem \ref{t1}. We recall that 
\begin{align}
\mathcal{S}_t g(x_0,\ell_0) &:= \int g(x,\ell) \left\{ - \frac12 \frac{(a(x_0,\ell_0)-a(x,\ell))}{a^2(x,\ell)}  H_3(t, \frac{|x|+|x_0|+ \ell-\ell_0}{\sigma(x,\ell)}) \right. \nonumber \\
& \quad \left. +  \frac{b(x_0,\ell_0) \mathrm{sign}(x_0)}{a^{\frac32}(x,\ell)} H_2(t, \frac{|x|+|x_0|+ \ell-\ell_0}{\sigma(x,\ell)}) \right\} \I_\seq{\ell> \ell_0} dx d\ell \label{def:S:local:time:drift} \\
& \quad + \int g(x,\ell_0) \left\{\frac12 (a(x_0,\ell_0) - a(x,\ell_0))  \left\{H_2(a(x,\ell_0) t, x-x_0) - H_2(a(x,\ell_0) t, x+x_0)\right\} \right. \nonumber \\
& \quad \quad \left. - b(x_0,\ell_0) \left\{H_1(a(x,\ell_0) t, x-x_0) - H_1(a(x,\ell_0) t, x+x_0)\right\}  \right\} \I_\seq{x x_0 \geq0} dx \nonumber
\end{align}

\noindent with $\mathrm{sign}(x_0)=-\I_\seq{x_0\leq 0} + \I_\seq{x_0>0}$. From this expression, we see that $(x_0,\ell_0) \mapsto \mathcal{S}_t g(x_0,\ell_0)$ is not continuous at zero, unless $b(0,\ell_0)=0$, $\ell_0 \in \rr_+$. Hence, we cannot use directly Corollary \ref{corolllary:semigroup}. We proceed as follows. We first consider the drift-less SDE. We then briefly indicate how to proceed in the presence of a bounded measurable drift by means of the Girsanov theorem. From now on, we let $b\equiv 0$. 

We will use the notation $\mathcal{S}_t g(x_0,\ell_0) = \int g(x,\ell)\hat \theta_t(x_0,\ell_0, x, \ell) \nu(x_0,\ell_0, dx,d\ell)$ with 
\begin{align}
\label{thetahat:local:time}
 \hat \theta_t(x_0,\ell_0, x, \ell)& := 
\begin{cases}
    - \frac12 \frac{(a(x_0,\ell_0)-a(x,\ell))}{a^2(x,\ell)}  H_3(t, \frac{|x|+|x_0|+ \ell-\ell_0}{\sigma(x,\ell)}),  & \ell> \ell_0,\\
      \frac12 (a(x_0,\ell_0) - a(x,\ell_0))  \left\{H_2(a(x,\ell_0) t, x-x_0) - H_2(a(x,\ell_0) t, x+x_0)\right\}, &     \ell=\ell_0.
\end{cases}
\end{align}





Since the function $(x_0,\ell_0) \mapsto \mathcal{S}_t g(x_0,\ell_0)$ is continuous on $\rr \times \rr_+$, applying Corollary \ref{corolllary:semigroup}, we get
\begin{align}
P_T g(x_0,\ell_0) & = \hat{P}_T g(x_0,\ell_0) + \sum_{n \geq 1}  \int_{\Delta_n(T)} \hat{P}_{s_n} \mathcal{S}_{s_n - s_{n-1}} \cdots \mathcal{S}_{T-s_1} g(x_0,\ell_0) \, d\s_{n} \label{semigroup:series:local:time}
\end{align}

\noindent with the convention $s_0=T$.

 In order to retrieve the transition density associated to $(P_t)_{t\geq 0}$, we are aiming to prove an integral representation for the above series. More precisely, our aim is to prove that the right-hand side of \eqref{semigroup:series:local:time} can be written as $\int g(x,\ell) p_T(x_0,\ell_0,x,\ell)\nu(x_0,\ell_0, dx, d\ell)$, with an explicit representation for $p_T(x_0,\ell_0,x,\ell)$. We start by an examination of the $n$-th term of the series expansion. 

Before proceeding, we observe that the measure $\nu(x_0,\ell_0, dx, d\ell)$ satisfies a useful convolution type property in the sense that 
\begin{gather}
\nu(x_0, \ell_0, dx',d\ell') \nu(x',\ell', dx, d\ell) =  u(x_0, \ell_0, x, \ell, dx',d\ell' ) \nu(x_0, \ell_0, dx, d\ell) \label{nulocal1}
\end{gather}

\noindent where we set 
\begin{align*}
u(x_0, \ell_0,x, \ell,dx',d\ell')
& := \begin{cases}
\I_{\{\ell_0<\ell'<\ell\}}dx'd\ell' + \I_{\{x'x_0> 0\}}dx' \delta_{\ell_0}(d\ell') + \I_{\{xx'>0\}}dx' \delta_{\ell}(d\ell'), & \,\, \ell_0 < \ell, \\
 \I_{\{ x_0 x'>0\}} dx' \delta_{\ell_0}(d\ell'),  &  \,\, \ell_0 = \ell,\,  x_0 x>0.
\end{cases}
\end{align*}

Applying repeatedly \eqref{nulocal1} and using Fubini's theorem, we get
\begin{align*}
& \int_{\Delta_n(T)} d\s_n\hat{P}_{s_n} \mathcal{S}_{s_{n-1}-s_n} \dots \mathcal{S}_{T-s_1}g(x_0,\ell_0) \\
& = \int_{\rr \times \rr_+} g(x,\ell) \left\{ \int_{\Delta_n(T)} d\s_n \int_{(\rr\times \rr_+)^{n}}\hat p_{s_n}(x_0,\ell_0, x_1, \ell_1) \right.\\
& \quad 	\left. \times \left[\prod_{i=1}^{n} \hat{\theta}_{s_{n-i}-s_{n-i+1}}(x_{i}, \ell_{i}, x_{i+1}, \ell_{i+1}) u(x_{0},\ell_{0}, x_{i+1},\ell_{i+1}, dx_{i},d\ell_{i})\right]\right\} \nu(x_0,\ell_0, dx, d\ell)\\
& =  \int_{\rr \times \rr_+} g(x,\ell) p^{n}_T(x_0,\ell_0,x,\ell)\nu(x_0,\ell_0, dx, d\ell)
\end{align*}

\noindent where we set  
\begin{align}\label{pn:local:time}
p^n_T(x_0,\ell_0, x, \ell) & := 
 \begin{cases}
 \int_{\Delta_n(T)} d\s_n \int_{(\rr\times \rr_+)^{n}}\hat p_{s_n}(x_0,\ell_0, x_1, \ell_1)\times \\
\left[\prod_{i=1}^{n} \hat{\theta}_{s_{n-i}-s_{n-i+1}}(x_{i}, \ell_{i}, x_{i+1}, \ell_{i+1}) u(x_{0},\ell_{0}, x_{i+1},\ell_{i+1}, dx_{i},d\ell_{i})\right]  & n \geq 1, \\
\hat p_{T}(x_0,\ell_0, x,\ell) & n = 0.
\end{cases}
\end{align}

From \eqref{thetahat:local:time} and the space-time inequality, it is easy to see that obtain the following estimate
\begin{align}
\label{bound:thetahat:local:time}
 |\hat \theta_t(x_0,\ell_0, x, \ell)|& \leq 
\begin{cases}
     \frac{C}{t^{\frac{3-\eta}{2}}}  H_0(ct, |x|+|x_0|+ \ell-\ell_0),  & \ell> \ell_0,\\
      \frac{C}{t^{1-\frac{\eta}{2}}}  H_0(c t, x-x_0), &     \ell=\ell_0,
\end{cases}
\end{align}

\noindent for some constants $C,\, c>1$.

We are ready to give a representation for the density of the couple $(X^{x_0}_t, \ell_0+ L^{0}_t(X^{x_0}
))$. As already mentioned in the introduction, we point out that the proof of the convergence of the asymptotic expansion for the transition
density is not standard in the current setting. Indeed, in the classical diffusion setting, the parametrix expansion of the transition density converges since the order of the singularity in time induced by the parametrix kernel $\hat{\theta}_t$ is of order $t^{-1+\frac{\eta}{2}}$, which is still integrable near $0$. The situation here is much more delicate. At first glance, the order of the time singularity in $\hat{\theta}_t$ consists in two parts. The first part corresponds to the non-singular part of the law of $(\bar{X}^{x_0}_t,\ell_0+L^{0}_t(\bar{X}^{x_0}))$. From \eqref{bound:thetahat:local:time}, it induces a singularity in time of order $t^{\frac{-3+\eta}{2}}$ which is integrable in time after integrating the kernel $H_0(ct, |x|+|x_0|+ \ell-\ell_0)$ on the domain $\rr \times (\ell_0,\infty)$. 

The second part corresponds to the singular case where the law of $(\bar{X}^{x_0}_t,\ell_0+L^{0}_t(\bar{X}^{x_0}))$, is the one of the proxy process killed when it reaches zero, and is absolutely continuous with respect to the singular measure $dx \delta_{\ell_0}(d\ell)$. Here the situation is standard and the singularity in time appearing in \eqref{bound:thetahat:local:time} is integrable. 

The main difficulty appears when one wants to control the whole convolution appearing in the right-hand side of \eqref{pn:local:time}. More precisely, it lies in the cross-terms which are of a different nature, for instance when one convolutes the non-singular part in the convolution kernel $\hat{\theta}_{T-s_1}$ and the singular part in the convolution kernel $\hat{\theta}_{s_1-s_2}$. Standard arguments such as the one used in \cite{frie:64} or \cite{Menozzi} do not guarantee the convergence of the integral defining \eqref{pn:local:time}. To overcome this difficulty and show that the parametrix expansion for the transition density converges, one has to make use of the key estimate obtained in Lemma \ref{estimate:kernel:proxy} which relies on the symmetry in the initial and terminal point of the density of the killed proxy process, in order to retrieve the integrability in time of the underlying convolution kernel.

As our second main result, we prove that the transition density of $(X_t, \ell_0+ L^{0}_t(X))_{t\geq0}$ exists and satisfies a Gaussian upper bound. Its proof is given in the Appendix.

\begin{theorem}\label{existence:bivariate:density:local:time} Assume that \A{R-$\eta$} and \A{UE} hold for some $\eta \in (0,1]$. For $(x_0,\ell_0) \in \rr\times \rr_+$, define the measure
\begin{align*}
p_T(x_0,\ell_0, dx, d\ell) & : = p_T(x_0,\ell_0, x, \ell)\nu(x_0, \ell_0, dx, d\ell)\\
					 & = p^{1}_T(x_0,\ell_0, x, \ell) \, \I_{\{ \ell_0 < \ell \}}  dx d\ell +  p^{2}_T(x_0,\ell_0, x) \, \I_\seq{x x_0 \geq 0} dx \delta_{\ell_0}(d\ell)
\end{align*} 
\noindent with $p_T(x_0,\ell_0, x, \ell) := \sum_{n\geq 0} p^{n}_T(x_0,\ell_0, x,\ell)$ and 
\begin{align*}
p^{1}_T(x_0, \ell_0, x, \ell)  :=\sum_{n\geq 0} p^{n}_T(x_0,\ell_0, x, \ell), \quad p^{2}_T(x_0, \ell_0, x) :=  \sum_{n\geq 0} p^{n}_T(x_0,\ell_0, x, \ell_0).
\end{align*}

Then, both series defining $p^{1}_T(x_0,\ell_0, x, \ell)$ and $p^{2}_T(x_0,\ell_0, x)$ converge absolutely and uniformly for $(x_0,\ell_0), \,(x, \ell)\in (\rr \times \rr_+)^2$. Moreover for $h \in \mathcal{C}_b(\rr \times \rr_+)$ the following representation for the semigroup holds,
\begin{equation*}
P_Th(x_0,\ell_0) = \int_{\rr \times \rr_+}  \,h(x,\ell) \, p_T(x_0,\ell_0,x,\ell)\nu(x_0, \ell_0, dx, d\ell).
\end{equation*}

Therefore, for all $(x_0,\ell_0)\in \rr\times \rr_+$, the function $(x,\ell) \mapsto p_T(x_0,\ell_0,x,\ell)$ is the probability density function of the random vector $(X^{x_0}_{T}, \ell_0+ L^{0}_T(X^{x_0}))$ with respect to the $\sigma$-finite measure $\nu(x_0,\ell_0, dx,d\ell)$, where $X^{x_0}_T$ is the solution taken at time $T$ of the SDE \eqref{sde:dynamics:local:time} starting from $x_0$ at time $0$, $L^{0}_T(X^{x_0})$ being its running symmetric local time at time $T$. 

Finally, there exists some constants $C,c>1$ such that for all $(x_0,\ell_0),\, (x,\ell)\in  \rr \times \rr_+$, the following Gaussian upper-bounds hold
\begin{equation}
\label{gaussian:upper:bound:local:time}
 p^{1}_T(x_0,\ell_0,x,\ell)   \leq C T^{-1/2} H_0(c T, |x|+|x_0|+\ell-\ell_0)\ \ \mbox{ and } \ \ p^{2}_T(x_0,\ell_0,x)   \leq  C H_0(c T, x-x_0).
\end{equation}

\end{theorem}

\begin{remark}\label{bounded:measurable:drift}The above Theorem proves the existence of a transition density for the Markov process $(X_t,\ell_0+L^{0}_t(X))_{t\geq0}$ where the dynamics is given by \eqref{sde:dynamics:local:time} without drift. In order to add a drift, one can use the Girsanov Theorem as follows. Let $\rr \times \rr_+ \ni (x,\ell) \mapsto b(x,\ell)$ be a real-valued bounded measurable function. We consider the unique weak solution $\left\{(X,W), (\Omega,\mathcal{F}, \P), ( \mathcal{F}_t)_{t\geq 0}\right\}$ of \eqref{sde:dynamics:local:time}. Let $\widetilde{W}_t := W_t + \int_0^t \tilde{b}(X_s,L^{0}_s(X)) ds$, with $\tilde{b}(x,\ell) := b(x,\ell)/\sigma(x,\ell)$. Then, defining the new probability measure on $\F_T$ by 
\begin{gather*}
\frac{d\mathbb{Q}}{d\P}:=\exp\left\{-\int_0^T \tilde{b}(X_s,L^{0}_s(X)) dW_s - \frac{1}{2}\int_0^T \tilde{b}^2(X_s,L^{0}_s(X)) ds\right\},
\end{gather*}

\noindent from Girsanov's Theorem, we know that $\big\{(\widetilde{X},\widetilde{W}), (\Omega,\mathcal{F},\mathbb{Q}), (\mathcal{F}_t )_{t\geq 0}\big\}$, with $\widetilde{X}_t = x + \int_0^t \sigma(\widetilde{X}_s,A_s(\widetilde{X})) d\widetilde{W}_s$ is a weak solution to \eqref{sde:dynamics:local:time} with $b\equiv 0$. We also know from Theorem \ref{weak:existence:local:time} that weak uniqueness holds for this equation and from Theorem \ref{existence:bivariate:density:local:time} it admits a transition density with respect to the measure $\nu(x_0,\ell_0, dx, d\ell)$. Therefore, for any bounded measurable function $h: \rr \times \rr_+ \rightarrow \rr$, we can write for $t\leq T$,
\begin{align*}
P_t h(x_0,\ell_0) & :=\E_{\P}[h(X_t,\ell_0+L^{0}_t(X))] \\
& = \int_{\rr \times \rr_+} h(x,\ell) \E_{\mathbb{Q}}\Big[\frac{d\mathbb{P}}{d\Q} \,\Big|\,(X_t,\ell_0+L^{0}_t(X))=(x,\ell)\Big] \,\bar{p}_t(x_0,\ell_0,x,\ell) \nu(x_0, \ell_0, dx,d\ell).
\end{align*}

From the above we deduce that the random vector $(X_t,\ell_0+L^{0}_t(X))$, $t>0$, admits a density with respect to the measure $\nu(x_0,\ell_0, dx,d\ell)$. Then, by an approximation argument that we omit, one may extend Theorem \ref{t1} to include functions $g$ that are bounded measurable with respect to the space variable $x$. Hence, one may iterate the first step formula and obtain the semigroup expansion \eqref{param:series} where the function $\mathcal{S}_t g$ is given by \eqref{def:S:local:time:drift} and $\hat{\theta}_t$ is defined accordingly. Finally, one has to repeat the arguments employed in the proof of Theorem \ref{existence:bivariate:density:local:time}. We omit the remaining technical details.
\end{remark}

\begin{remark}
We again point out that the representation in infinite series obtained previously are of great interest. For instance, one may be interested in studying regularity properties of $(t,x_0,\ell_0, x,\ell) \mapsto p^{1}_t(x_0,\ell_0,x,\ell), p^{2}_t(x_0,\ell_0,x,\ell)$, to derive a probabilistic interpretation of $P_t g$ and $p_t(x_0,\ell_0,x,\ell)$ and to obtain some integration by parts formulas or an unbiased Monte Carlo simulation scheme. We refer e.g. to \cite{frikha:li:kohatsu}, \cite{garroni:menaldi} and the references therein for some results in that direction concerning to the first hitting times of one-dimensional elliptic diffusions.
\end{remark}
\subsection{A diffusion process with coefficients depending on its running maximum} $\,$\vskip5pt
We now turn our attention to the following SDE with dynamics
\begin{equation}
\label{sde:dynamics:maximum}
X_t = x_0 + \int_0^t b(X_s, M_s) ds + \int_0^t \sigma(X_s, M_s) dW_s
\end{equation}

\noindent where $M_t = A_t(X) := m_0\vee \max_{0\leq s\leq t} X_s$, $m_0\geq x_0$, is the running maximum of the process $X$ at time $t$. The state space of the process $(X_t,M_t)_{t\geq0}$ is denoted by the closed set $\mathcal{J}=\left\{(x,m) \in \rr^2: x\leq m \right\}$. We define accordingly the collection of linear maps $P_t f(x_0, m_0) = \mathbb{E}[f(X_t,M_t)]$ for $f\in \mathcal{B}_b(\mathcal{J})$ and introduce the following assumptions:
\smallskip

\begin{itemize}

\item[\A{R-$\eta$}] The coefficients $b$ and $a=\sigma^2$ are bounded measurable functions defined on $\mathcal{J}$. The diffusion coefficient $a$ is $\eta$-H\"older continuous on $\mathcal{J}$.  

\item[\A{UE}] There exists some constant $\underline{a}>0$ such that $\forall (x,m)\in \mathcal{J}$, $\underline{a} \leq a(x,m)$.

\end{itemize}
\smallskip

Although the lines of reasoning used in the proof of Theorem \ref{bivariate:density:maximum} and Theorem \ref{existence:bivariate:density:max} are rather similar to those employed in the case of the SDE with its running symmetric local time, we decided to include this example in order to illustrate the generality of our framework. Indeed, unlike the local time and also the examples considered so far in the literature by means of the parametrix technique, see e.g. \cite{Menozzi}, the running maximum is not a continuous additive functional. This difference is reflected in the definition of the collection of linear maps $(P_t)_{t\geq0}$, and also in the definition of the approximation process $\bar Y$ and the maps $(\bar{P}_t)_{t\geq0}$, see below in Section \ref{weakuniqueness:maximum}. This allows to define a Markov semigroup for $\bar{Y}$ and later on for $Y$.

\smallskip 
The weak existence of a solution to the SDE \eqref{sde:dynamics:maximum} follows from standard results, see e.g. Chapter 6 in \cite{stro:vara:79} or Theorem 5.4.22 and Remark 5.4.23 in \cite{kara:shre:91}, see also Forde \cite{Forde20112802} for another approach under the assumption that $a$ and $b$ are continuous bounded functions. We also refer to the previous Remark \ref{weak:existence:bounded:mes:drift} for the case of bounded measurable drift.  

In the same spirit as in the previous section, we can characterise solutions of the SDE \eqref{sde:dynamics:maximum} in terms of the associated (local) martingale problem. We let $\mathcal{D}$ be the class of functions $f:\mathcal{J} \rightarrow \rr$ such that $f \in \mathcal{C}^{2,1}_b(\mathcal{J})$ which satisfies the condition $\partial_2 f(m, m)= 0$, $m\in \rr$. For $f\in \mathcal{D}$, we define the operator 
$$\mathcal{L} f(x,m)= \frac12 a(x,m) \partial^2_1 f(x, m)+ b(x,m) \partial_1f(x,m).$$ 
Observing that the process $t \mapsto M_t$ increases only on the set $\left\{t: X_t = M_t \right\}$ and by applying It\^o's lemma, we get
\begin{equation}
\label{ito:formula:maximum}
f(t,X_t,M_t) = f(0,x_0,m_0) + \int_0^t \left\{ (\partial_1 f(s,.) + \mathcal{L}f(s,.))(X_s, M_s) \right\} ds + \int_0^t \sigma(X_s,M_s) \partial_2 f(s,X_s,M_s) dW_s 
\end{equation}

\noindent for $f\in \mathcal{C}^1(\rr_+, \mathcal{D})$. 

\subsection{Weak uniqueness and representation of the transition density}\label{weakuniqueness:maximum}$\,$\vskip5pt

We now introduce the proxy process $\bar{X}_t:=x_0 + \sigma(z_1)W_t$, $t\geq0$, obtained from the original process $X$ by removing the drift part and by freezing the diffusion coefficient at $z_1=(x_1,m_1)\in \mathcal{J}$ in the dynamics \eqref{sde:dynamics:maximum}. As already done for the original process $Y$, for $f\in \mathcal{C}_b(\mathcal{J})$, we define accodingly 
\begin{align*}
 \bar{P}_t f(x_0,m_0) & = \E[f(\bar{X}_t, \bar{M}_t)] \\
 &  = \E[f(x_0 + \sigma(z_1) W_t, m_0 \vee \max_{0\leq s\leq t}(x_0 + \sigma(z_1) W_s) )],
\end{align*}

\noindent and, with $\bar \sigma = \sigma(z_1)$, $\bar a = \bar \sigma ^2$, for $f\in \mathcal{D}$, the operator
\begin{align*}
\bar{\mathcal{L}}f(x,m) & =  \frac12 \bar{a} \partial^2_1 f(x, m), \quad (x,m) \in \mathcal{J}.
\end{align*}

We now compute the law of the couple $(\bar{X}_t, \bar{M}_t)$. We denote by $T_{m_0}=\inf\left\{ t\geq0: x_0 + \bar{\sigma} W_t = m_0\right\}$ the first hitting time of $0$ by the process $(\bar{X}_t)_{t\geq0}$. Let $f\in \mathcal{C}_b(\mathcal{J})$. We decompose $\bar{P}_t f(x_0,m_0)$ as follows
\begin{align*}
\bar{P}_t f(x_0,m_0) & := \mathbb{E}[f(x_0+ \bar{\sigma} W_t, m_0) \I_\seq{T_{m_0}\geq t}] \\
& \quad + \mathbb{E}[f(x_0+ \bar{\sigma} W_t, m_0 \vee \max_{0\leq s \leq t}(x_0 + \bar{\sigma} W_s)) \I_\seq{T_{m_0} < t}]\\
& = I + II.
\end{align*}

From the reflection principle of Brownian motion, see e.g. \cite{kara:shre:91}, one derives the law of the proxy process killed when it exits $(-\infty, m_0)$, namely
$$
 I = \E[f(x_0+ \bar{\sigma} W_t, m_0) \I_\seq{T_{m_0}\geq t}]  = \int_{\rr^2} f(x,m) \left\{H_0(\bar{a} t, x-x_0) - H_0(\bar{a} t, 2m_0-x-x_0) \right\} \I_\seq{x\vee x_0 \leq m_0} \, dx \delta_{m_0}(dm)
$$
From the bivariate density of $(W_t, \max_{0\leq s \leq t}W_s)$, see e.g. Proposition 2.8.1 in  \cite{kara:shre:91}, we obtain
\begin{align*}
II & = \E[f(x_0+ \bar{\sigma} W_t, m_0 \vee \max_{0\leq s \leq t}(x_0 + \bar{\sigma} W_s)) \I_\seq{T_{m_0} < t}] \\
       & =  \E[f(x_0+ \bar{\sigma} W_t, \max_{0\leq s \leq t}(x_0 + \bar{\sigma} W_s)) \I_\seq{\max_{0\leq s \leq t}(x_0 + \bar{\sigma} W_s) \geq m_0}] \\
       & = \int_{\rr^2} f(x,m) (-2H_1)(\bar{a} t, 2m-x-x_0)\I_\seq{x \vee x_0\vee m_0 \leq m} dx dm.
\end{align*}

Combining $I$ and $II$, we see that the couple $(\bar{X}_t, \bar{M}_t)$ admits a density $(x,m) \mapsto \bar{p}_t(x_0,m_0, x, m)$ that is
\begin{gather*}
\bar{p}_t(x_0, m_0, dx, dm)  = \bar{p}_t(x_0, m_0, x, m) \nu(x_0, m_0, dx, dm),  \quad t>0, 
\end{gather*}

\noindent with $\bar{p}_t(x_0, m_0, x, m) := \bar{f}_t(x_0,x) \I_\seq{m= m_0} + \bar{q}_t(x_0, m_0,  x, m)  \I_\seq{m_0 < m}$ and 
\begin{align*}
\bar{f}_t(x_0,x)&  := H_0(\bar{a} t, x-  x_0) - H_0(\bar{a} t, 2m_0-x-x_0), \\
\bar{q}_t(x_0, m_0, x, m) &:= -2H_1(\bar{a} t, 2m-x-x_0), \\
\nu(x_0,m_0,dx, dm) & :=   \I_\seq{x\leq m} \I_\seq{m_0 < m} dx dm + \I_\seq{x_0\leq m_0} \I_\seq{x\leq m_0}  dx\delta_{m_0}(dm).
\end{align*}	
Moreover, as already mentioned in assumption \A{H1} in Section \ref{perturb:section}, we let $\hat{p}_t(x_0,m_0, x,m) = \hat{f}_t(x_0,x) \I_\seq{m= m_0} + \hat{q}_t(x_0, \ell_0,  x, \ell)  \I_\seq{m_0 < m}$ with
\begin{align*}
\hat{f}_t(x_0,x)&  := H_0(a(x, m_0) t, x-x_0) - H_0(a(x, m_0)t, 2m_0- x -x_0), \\
\hat{q}_t(x_0,m_0, x, m) &:= -2H_1(a(x,m) t, 2m-x-x_0).
\end{align*}

We also observe that $(\bar{P}_t)_{t\geq0}$ defines a Markov semigroup. The first main result of this section establishes weak uniqueness for the SDE \eqref{sde:dynamics:maximum} by proving that assumptions \A{H1} and \A{H2} of Section \ref{perturb:section} are satisfied. Its proof is given in the Appendix.
\begin{theorem} \label{bivariate:density:maximum} For $\eta \in (0,1]$, under \A{R-$\eta$} and \A{UE}, weak uniqueness holds for the SDE \eqref{sde:dynamics:maximum}.
\end{theorem}

Next we show that, given $(x_0, m_0) \in \mathcal{J}$, the law of $ (X^{x_0}_T, M_T)$ is absolutely continuous with respect to the measure $\nu(x_0,m_0,dx, dm)$. Our strategy is similar to the one used in the case of the SDE with its running local time, that is we establish a representation in infinite series of $P_t g$ from which stems an explicit representation of the density of the couple $(X^{x_0}_t, \ell_0 + L^0_t(X))$, see Theorem \ref{existence:bivariate:density:max} below. We first set
\begin{align}
\mathcal{S}_t g(x_0,m_0) &:= \int g(x,m) \left\{\frac12 (a(x_0,m_0) - a(x,m))(-2H_3)(a(x,m)t,2m-x-x_0)\right.\nonumber \\
& \quad  + b(x_0,m_0) (-2H_2)(a(x,m)t,2m-x-x_0)\bigg\}\I_{\{x \leq m\}}\I_{\{m_0 < m\}}dx dm \nonumber \\
& \quad + \int g(x,m_0) \left\{\frac12 (a(x_0,m_0) - a(x,m_0))( H_2(a(x,m_0)t, x-x_0) - H_2(a(x,m_0)t,2m_0-x-x_0) )\right.\label{semigroup:max:s} \\\nonumber 
& \quad + b(x_0,m_0)  ( H_1(a(x,m_0)t,2m_0-x-x_0) - H_1(a(x,m_0)t, x-x_0) )\bigg\}\I_{\{x< m_0\}} dx \\
& = \int g(x, m) \hat{\theta}_t(x_0,m_0, x, m) \nu(x_0,m_0, dx, dm) \nonumber
\end{align}

\noindent with 
\begin{align*}
&\hat{\theta}_t(x_0, m_0, x, m)  :=\\
& 
\begin{cases}
     \frac12 (a(x_0,m_0) - a(x, m))(-2H_3)(a(x,m)t,2m-x-x_0) + b(x_0,m_0) (-2H_2)(a(x,m) t,2m-x-x_0),  & x \leq m,\,  m_0 < m,\\
      \frac12 (a(x_0,m_0) - a(x,m_0))( H_2(a(x,m_0)t, x-x_0) - H_2(a(x,m_0)t,2m_0-x-x_0) ) & \, x< m_0, \, m=m_0. \\
         + b(x_0,m_0)  ( H_1(a(x, m_0) t, 2m_0-x-x_0) - H_1(a(x, m_0)t, x-x_0) ),
\end{cases}
\end{align*}

 We remark that the function $(x_0, m_0) \mapsto \mathcal{S}_t g(x_0, m_0)$ is continuous on $\mathcal{J}$. This is different from the case of the local time, where due to the presence of the sign function, we required that $b=0$ in order to ensure that the continuity of $(x_0, m_0) \mapsto \mathcal{S}_t g(x_0, m_0)$ on $\mathcal{J}$. 

One may now iterate the first step of the expansion obtained in Theorem \ref{t1}. More precisely, by applying Corollary \ref{corolllary:semigroup} and setting by convention $s_0=T$, we get
\begin{align}
P_T g(x_0, m_0) & = \hat{P}_T g(x_0, m_0) + \sum_{n \geq 1}  \int_{\Delta_n(T)} \hat{P}_{s_n} \mathcal{S}_{s_n - s_{n-1}} \cdots \mathcal{S}_{T-s_1} g(x_0, m_0) \, d\s_{n}. \label{semigroup:series:max}
\end{align}

We again observe the following convolution type property of the singular measure, namely
\begin{gather}
\nu(x_0,m_0, dx',dm') \nu(x',m', dx,dm) =  u(x_0,m_0,x,m,dx',dm' )\nu(x_0,m_0,dx,dm) \label{numax1}
\end{gather}
where we set 
\begin{align*}
& u(x_0,m_0,x,m,dx',dm')\\
&:= \begin{cases}
\I_{\{x'\vee x_0\vee m_0<m'\}}\I_{\{m'<m\}} dx'dm' + \I_{\{x'<m_0\}}dx'\delta_{m_0}(dm') + \I_{\{x'<m\}}dx'\delta_{m}(dm') & \,\, m_0 < m, x\leq m, \\
\I_{\{x' < m_0\}}dx'\delta_{m_0}(dm') &  \,\, m = m_0, x\leq m_0.
\end{cases}
\end{align*}
Then we examine the $n$-th term of the series expansion, and by using Fubini's theorem and recursively applying \eqref{numax1}, it can be expressed as 
\begin{align*}
& \int_{\Delta_n(T)} d\s_n\hat{P}_{s_n} \mathcal{S}_{s_{n-1}-s_n} \dots \mathcal{S}_{T-s_1}g(x_0, m_0) \\
& = \int_{\rr^2} g(x, m) \left\{ \int_{\Delta_n(T)} d\s_n \int_{(\rr^2)^{n}}\hat p_{s_n}(x_0, m_0, x_1, m_1) \right.\\
& \quad \left.\times \left[\prod_{i=1}^{n} \hat{\theta}_{s_{n-i}-s_{n-i+1}}(x_{i}, m_{i}, x_{i+1}, m_{i+1}) u(x_{0}, m_{0}, x_{i+1}, m_{i+1}, dx_{i},dm_{i})\right]\right\} \nu(x_0, m_0, dx, dm)\\
& =  \int_{\rr^2} g(x,m) p^{n}_T(x_0, m_0,x, m)\nu(x_0, m_0, dx, dm)
\end{align*}

\noindent where we set  
\begin{align}\label{pn:max}
p^n_T(x_0, m_0, x, m) := & 
 \begin{cases}
 \int_{\Delta_n(T)} d\s_n \int_{(\rr^2)^{n}}\hat p_{s_n}(x_0, m_0, x_1, m_1)\times \\
\left[\prod_{i=1}^{n} \hat{\theta}_{s_{n-i}-s_{n-i+1}}(x_{i}, m_{i}, x_{i+1}, m_{i+1}) u(x_{0}, m_{0}, x_{i}, m_{i}, dx_{i+1},dm_{i+1})\right] & n \geq 1, \\
\hat p_{T}(x_0,m_0, x,m) & n = 0.
\end{cases}
\end{align}

We are now ready to give a representation for the density of the couple $(X^{x_0}_t, M_t(X^{x_0}))$. As already mentioned in the case of the SDE with its running local time, the proof of the convergence of the asymptotic expansion for the transition density is not standard in the current setting. To overcome the main difficulty which again comes from the different nature of the two kernels in $\hat{\theta}_t$ one has to make use of the key estimate obtained in Lemma \ref{estimate:kernel:proxy}. This allows to obtain the convergence of the parametrix expansion for the transition density. Similar Gaussian upper bounds for this density are also established. The proof is given in the Appendix.

\begin{theorem}\label{existence:bivariate:density:max} Assume that \A{R-$\eta$} and \A{UE} hold for some $\eta \in (0,1]$. For $(x_0,m_0) \in  \mathcal{J}$, define the measure
\begin{align*}
p_T(x_0,m_0, dx, dm) & : = p_T(x_0,m_0, x, m)\nu(x_0,m_0,dx,dm)\\
					 & = p^{1}_T(x_0,m_0, x, m) \,  \I_\seq{x\leq m} \I_\seq{m_0 < m} dx dm +  p^{2}_T(x_0,m_0, x)\I_\seq{x\leq m_0}  \, dx \,\delta_{m_0}(dm)
\end{align*} 
\noindent with $p_T(x_0,m_0, x, m) := \sum_{n\geq 0} p^{n}_T(x_0, m_0, x, m)$ and 
\begin{align*}
p^{1}_T(x_0,m_0, x, m)  := \sum_{n\geq 0} p^{n}_T(x_0, m_0, x, m), \quad p^{2}_T(x_0,m_0, x) := \sum_{n\geq 0} p^{n}_T(x_0, m_0, x, m_0).
\end{align*}

Then, both series $p^{1}_T(x_0,m_0, x, m)$ and $p^{2}_T(x_0,m_0, x)$ converge absolutely and uniformly for $(x_0,m_0), (x, m)\in \mathcal{J}$. Moreover for $h \in \mathcal{C}_b(\mathcal{J})$, the following representation for the semigroup holds,
\begin{equation*}
P_Th(x_0,m_0) = \int_{\rr \times \rr_+}  \,h(x,m) \, p_T(x_0, m_0, dx, dm).
\end{equation*}

Therefore, for all $(x_0,m_0)\in \mathcal{J}$, the function $\mathcal{J} \ni (x,m) \mapsto p_T(x_0,m_0, x,m)$ is the probability density function with respect to the measure $\nu(x_0,m_0, dx, dm)$ of the random vector $(X^{x_0}_{T},m_0\vee M_T(X^{x_0}))$, where $X^{x_0}_T$ is the solution taken at time $T$ of the SDE \eqref{sde:dynamics:maximum} starting from $x_0$ at time $0$ and $M_T(X^{x_0})$ is its running maximum at time $T$ starting from $m_0$ at time $0$. 

Finally, for some positive $C,c>1$, for all $(x_0,m_0), (x,m)\in  \mathcal{J}$, the following Gaussian upper-bounds hold
\begin{equation}
\label{gaussian:upper:bound:max}
 p^{1}_T(x_0,m_0,x,m)   \leq C T^{-1/2} H_0(c T, 2m-x-x_0)\ \ \mbox{ and } \ \ p^{2}_T(x_0,\ell_0, x)   \leq  C H_0(c T, x-x_0).
\end{equation}

\end{theorem}

\vskip15pt
\section{Conclusion}
In this paper, we obtained weak existence and uniqueness for some SDEs with coefficients depending on some path-dependent functionals $(A_t(X))_{t\geq0}$ under mild assumptions on the coefficients, namely bounded measurable drift and uniformly elliptic H\"older-continuous diffusion coefficient. 
We illustrated our approach on two examples: an SDE with coefficients depending on its running local time and an SDE with coefficients depending on its running maximum. We also established the existence as well as a representation in infinite series of the density for the couple $(X_t,A_t(X))$ in both examples. Some Gaussian upper-bounds are also obtained. 

Obviously, a wide variety of Brownian functionals can be investigated. Simple extensions include for instance the case $A_t(X)=(\tau_L \wedge t, X_{\tau_L \wedge t})$, where $\tau_L=\inf\left\{ t\geq 0:  X_{t}\geq L \right\}$ is the first hitting time of the barrier $L$ by $X$ or the bivariate functional $A_t(W)=(\min_{0\leq s\leq t}X_s, \max_{0\leq s \leq t} X_s)$. More challenging extensions could include other type of processes. One notably may consider the case of a skew diffusion with path-dependent coefficients involving its local and occupation times, see Appuhamillage \& al. \cite{appuhamillage2011} for an expression of the trivariate density $(B^{(\alpha)}_t, L^{0}_t(B^{(\alpha)}), \Gamma^0_t(B^{(\alpha)}))$, $t\geq0$, where $(B^{(\alpha)}_t)_{t\geq0}$ is an $\alpha$-skew Brownian motion or reflected SDEs. This will be developed in future works.

\subsection*{Acknowledgement}
The authors wish to thank Professor Arturo Kohatsu-Higa for his careful readings and valuable comments on the writing of this paper.

\vskip15pt
\section{Appendix}

\subsection{Proof of Theorem \ref{weak:existence:local:time}}\label{first:part:appendix}\rule[-10pt]{0pt}{10pt}\\

From the Markov property of the Brownian motion $W$, one can deduce that $(\bar X_t, \ell_0+L^{0}_t(\bar X))_{t\geq 0}$ is a Markov process so that
$$
\mathbb{E}[f(\bar{X}_s+ \bar \sigma (W_{t+s}-W_s), \ell_0 + L^{0}_{s}(\bar X) + L^{0}_{t+s}(\bar{X}) - L^{0}_s(\bar{X}))|\mathcal{F}_s]  = \bar{P}_t f(\bar{X}_s, \ell_0 + L^{0}_{s}(\bar{X})).
$$
From the expression of $\bar{p}_t$, direct computations show that if $f\in \mathcal{C}^{\infty}_b(\rr \times \rr_+)$, then one has $\bar{P}_t f \in \mathcal{D}$, $t>0$ and $\partial_t \bar{P}_t f = \bar{\mathcal{L}} \bar{P}_t f = \bar{P}_t \bar{\mathcal{L}} f$, for $f\in \mathcal{D}$. Indeed, the latter formula holds since $\bar{\mathcal{L}}$ is the infinitesimal generator of $(\bar{P}_t)_{t\geq0}$ acting on $\mathcal{D}\subset \mathrm{Dom}(\mathcal{\bar L})$. One can also obtain the same result by using the generalised It\^o's formula of Proposition \ref{generalized:ito:lemma}. 

\smallskip
Now, in order to obtain $\bar{\theta}_t(x_0,\ell_0,x,\ell)$, we first remark that 
\begin{align*}
\forall t>0, \quad (\mathcal{L} - \bar{\mathcal{L}}) \bar{P}_t f(x_0,\ell_0) & = \frac12 (a(x_0,\ell_0)-\bar{a})\partial^2_1 \bar{P}_t f(x_0,\ell_0)  + b(x_0,\ell_0)\partial_1 \bar{P}_t f(x_0,\ell_0) \\
& = \int f(x,\ell) \left\{ \frac12 (a(z_0)- a(z_1)) \partial^2_{x_0} \bar{p}_t(x_0-,\ell_0, dx, d\ell)  + b(z_0) \partial_{x_0} \bar{p}_t(x_0-,\ell_0, dx, d\ell) \right\} 
\end{align*}

\noindent for $(x_0,\ell_0) \in \rr \times \rr_+$, where we used the Lebesgue differentiation theorem for the last equality. Simple computations yield
\begin{align*}
\partial_{x_0} \bar{p}_t(x_0-,\ell_0, dx, d\ell) & = \left\{\partial_{x_0}\bar{f}_t(x_0,x)  \I_{\{\ell=\ell_0\}}  + \partial_{x_0} \bar{q}_t (x_0-,\ell_0, x, \ell)  \I_\seq{\ell_0  < \ell} \right\}\nu(x_0,\ell_0, dx, d\ell)  \, , \\
\partial^2_{x_0} \bar{p}_t(x_0-,\ell_0, dx, d\ell) & =   \left\{\partial^2_{x_0} \bar{f}_t(x_0,x) \I_{\{\ell=\ell_0\}}   + \partial^2_{x_0} \bar{q}_t (x_0-,\ell_0, x, \ell)  \I_\seq{\ell_0 < \ell}\right\} \nu(x_0,\ell_0, dx, d\ell)
\end{align*}

\noindent with,
\begin{align*}
\partial_{x_0} \bar{f}_t(x_0,x) & = - (H_1(\bar{a}t, x-x_0) + H_1(\bar{a}t, x+x_0) ), \\ 
\partial^2_{x_0} \bar{f}_t(x_0,x) & = (H_2(\bar{a}t, x-x_0) - H_2(\bar{a}t, x+x_0) ), \\ 
\partial_{x_0} \bar{q}_t(x_0-,\ell_0, x, \ell) & =  -\frac{\mathrm{sign}(x_0)}{\bar a^{\frac 3 2}} H_2(t, \frac{|x|+|x_0|+ \ell-\ell_0}{\bar \sigma}), \\
\partial^2_{x_0} \bar{q}_t(x_0-,\ell_0, x, \ell) & =  -\frac{1}{\bar a^{2}} H_3(t, \frac{|x|+|x_0|+ \ell-\ell_0}{\bar \sigma}).
\end{align*}


Hence, one has $(\mathcal{L} - \bar{\mathcal{L}}) \bar{P}_t f(x_0,\ell_0) := \int f(x,\ell) \bar \theta^{z_1}_t(x_0,\ell_0, x, \ell) \nu(x_0,\ell_0, dx, d\ell)$ with 
\begin{align*}
& \bar \theta^{z_1}_t(x_0,\ell_0, x, \ell)  \\
& :=
\begin{cases}
    - \frac12 \frac{(a(z_0)-a(z_1))}{\bar a ^2}  H_3(t, \frac{|x|+|x_0|+ \ell-\ell_0}{\bar \sigma}) - b(z_0) \frac{\mathrm{sign}(x_0)}{\bar a^{\frac3 2}} H_2(t, \frac{|x|+|x_0|+ \ell-\ell_0}{\bar \sigma}),  &  \ell> \ell_0,\\
      \frac12 (a(z_0)-a(z_1))  \left\{H_2(\bar a t, x-x_0) - H_2(\bar{a}t, x+x_0)\right\} - b(z_0) \left\{H_1(\bar a t, x-x_0) + H_1(\bar{a}t,x+x_0)\right\},\, 
&       \ell=\ell_0.
\end{cases}
\end{align*}

We now prove that \A{H1} (iv) and (v) are satisfied. Before that, we point out that the inequality 
$$(x+x_0)^2=(x-x_0)^2+4x x_0\geq (x-x_0)^2$$
is valid on the set $\left\{ x x_0 \geq 0\right\}$ which implies $H_0(ct, x+x_0) \leq H_0(ct,x-x_0)$. We first prove \A{H1} (iv). Using the fact that $a$ is uniformly elliptic and bounded together with the space-time inequality 
%
%
we can bound $\bar\theta^{z_1}_t(x_0,\ell_0, x, \ell)$ as follows
\begin{align*}
|\bar\theta^{z_1}_t(x_0,\ell_0, x, \ell)|
& \leq 
\begin{cases}
C \left(\frac{1}{t^\frac{3}{2}}+ \frac{|b|_\infty}{t} \right) H_0(ct, |x_0| + |x|+ \ell-\ell_0) ,& \ell>\ell_0,  \\
C \left(\frac{1}{t}+  \frac{|b|_\infty}{t^{\frac{1}{2}}} \right) H_0(c t, x_0-x) ,  &  x_0x>0, \, \ell= \ell_0. 
\end{cases}
\end{align*}
Similar computation shows that $\hat\theta_t(x_0,\ell_0, x, \ell)$ can be bounded as follows
\begin{align*}
|\hat\theta_t(x_0,\ell_0, x, \ell)|
& \leq 
\begin{cases}
C \left( \frac{1}{t^{\frac{3-\eta}{2}}} + \frac{|b|_\infty}{t} \right)H_0(c t, |x_0| + |x|+ \ell-\ell_0) ,& \ell>\ell_0  \\
C \left( \frac{1}{t^{1-\frac{\eta}{2}}} +  \frac{|b|_\infty}{t^{\frac{1}{2}}} \right) H_0(c t, x_0-x),  &  x_0x>0, \, \ell= \ell_0. 
\end{cases}
\end{align*}

\noindent On the set $\{\ell > \ell_0\}$, the integral of $\hat\theta_t(x_0,\ell_0, x, \ell)$ against $\nu(x_0,\ell_0, dx,d\ell)$ can be estimated 
through integration by parts with respect to $\ell$ and using the fact that $t\leq T$, that is
\begin{align}
\left( \frac{1}{t^{\frac{3-\eta}{2}}} + \frac{|b|_\infty}{t} \right)\int_\rr  \int_{\ell_0}^\infty H_0(ct, |x_0| + |x|+ \ell-\ell_0)dxdl  & = C\frac{1}{t^\frac{3-\eta}{2}}\int_\rr  \int_{0}^\infty H_0(c t, |x_0| + |x| + \ell)d\ell dx\label{localtime_estimate_2}\\
																					  & = \frac{1}{t^\frac{3-\eta}{2}} \int_\rr  \int_{0}^\infty \ell (-H_1)(c t, |x_0| + |x| + \ell)d\ell dx \nonumber \\
  																					  & \leq C \frac{1}{t^\frac{3-\eta}{2}} \int_{0}^\infty\int^\infty_0   (\ell+|x_0|) (-H_1)(ct, |x_0| + x + \ell)dxd\ell \nonumber \\
  			    															          & = C\frac{1}{t^\frac{3-\eta}{2}} \int_{0}^\infty(\ell+|x_0|) H_0(ct, |x_0|  + \ell)d\ell \nonumber \\
  			    															          & \leq C\frac{|b|_\infty}{t^{1-\frac{\eta}{2}}}. \nonumber 
\end{align}
On the set $\left\{ x_0x>0 \right\}\cap \left\{\ell= \ell_0\right\}$, straightforward integration gives
\begin{align}
\int_\rr  \left(\frac{1}{t^{1-\frac{\eta}{2}}} + \frac{|b|_\infty}{t^{\frac{1}{2}}} \right)H_0(ct,x_0-x)dx
& \leq \frac{C(1+|b|_{\infty}t^{\frac{1-\eta}{2}})}{t^{1-\frac{\eta}{2}}}. \label{localtime_estimate_3}
\end{align}

From the above computations, we conclude that \A{H1} (iv) is satisfied. We now prove \A{H1} (v) that is,
\bde
\lim_{\varepsilon\rightarrow 0}\int_{\rr^2} f(x,\ell)\hat p_\varepsilon(x_0,\ell_0, dx,d\ell)  \rightarrow f(x_0,\ell_0), \quad f\in \mathcal{C}_b(\rr \times \rr_+).
\ede 
We consider the change of variable $\ell' = \ell-\ell_0$ and $x = x'$, and decompose $\int_{\rr^2} f(x,\ell)\hat p_\varepsilon(x_0,\ell_0, dx,d\ell)$ as follows
\begin{align*}
\int_{\rr\times [0,\infty)} f(x',\ell_0+ \ell')[\hat p_\varepsilon(x_0,0, dx',d\ell')- \bar p_\varepsilon(x_0,0, dx',d\ell')] + \int_{\rr\times [0,\infty)} f(x',\ell_0+ \ell')\bar p_\varepsilon(x_0,0,dx',d\ell') 
\end{align*}

\noindent where the frozen point in $\bar p_\varepsilon$ is given by $(x_1,\ell_1) = (x_0,0)$. It is clear from the continuity of $f$ that the second term converges to $f(x_0,\ell_0)$ as $\varepsilon \downarrow 0$. To show that the first term vanishes as $\varepsilon \downarrow 0$, we apply the mean value theorem
\begin{align*}
&|\int_{\rr\times [0,\infty)} f(x',\ell_0+ \ell')[\hat p_\varepsilon(x_0,0,dx',dl')- \bar p_\varepsilon(x_0,0,dx',d\ell')]| \\
&\leq C|f|_\infty \int_{\rr\times [0,\infty)}\varepsilon (|x_0-x'|^\eta + |\ell'|^\eta) \left\{|H_3(c \varepsilon, |x_0|+|x|+\ell')| + |H_2(c\varepsilon, |x_0| + |x|+ \ell')| \right\}dx'd\ell' \\
& \quad + C|f|_\infty \int_\rr \varepsilon|x_0-x'|^\eta [|H_2(c \varepsilon, x_0-x')| + |H_2(c \varepsilon, x_0 + x')|]\I_\seq{x_0x'>0}dx'\\
& \leq C\varepsilon^{\frac \eta 2}
\end{align*}

\noindent where the last inequality follows from the space-time inequality and computations similar to \eqref{localtime_estimate_2} and \eqref{localtime_estimate_3}.

\vskip 10pt 
We now prove that \A{H2} holds. Let $g\in \mathcal{D}$. Using the expression of the measure $(x_0,\ell_0) \mapsto \bar{p}_t(x_0,\ell_0, dx, d\ell)$, we obtain that $(x_0,\ell_0) \mapsto \bar{P}_t g(x_0,\ell_0) \in \mathcal{C}^{2,1}_b(\rr \backslash{\left\{0\right\}}, \rr_+, \rr)$ and some simple computations (that we omit) shows that 
$$
\frac{\partial_1 \bar P_t g(0+,\ell_0) - \partial_1 \bar P_t g(0-,\ell_0)}{2} = \int g(x,\ell_0) (-H_1)(\bar{a}t,|x|) dx - \int \frac{g(x,\ell) }{\bar a^{\frac 3 2}} H_2(t, \frac{|x|+\ell-\ell_0}{\bar \sigma}) \I_\seq{\ell \geq \ell_0}dx d\ell = - \partial_2 \bar P_t g(0,\ell_0).
$$

Moreover, using integration by parts formula, one shows that $(x_0,\ell_0) \mapsto \partial_1 \bar P_t g(x_0,\ell_0), \, \partial^2_1\bar P_t g(x_0,\ell_0), \, \partial_2 \bar P_t g(x_0,\ell_0)$ are bounded by a constant depending of $|\partial_1 g|_\infty,\, |\partial^2_1 g|_\infty, \, |\partial_2 g|_\infty$ which is uniform in $t$. Hence, one has $\bar R_\lambda \mathcal{D} \subset \mathcal{D}$ and $(\mathcal{L} - \bar{\mathcal{L}}) \bar R_\lambda g = \int_0^{\infty} e^{-\lambda t} (\mathcal{L} - \bar{\mathcal{L}}) \bar{P}_t g dt$ for $g \in \mathcal{D}$. The relation $(\lambda - \bar{\mathcal{L}}) \bar R_\lambda g = \bar{R}_\lambda (\lambda - \bar{\mathcal{L}}) g$ follows from the fact that $ \bar{P}_t \bar{\mathcal{L}} g = \partial_t \bar{P}_t g$ which in turn is a consequence of the generalised It\^o lemma obtained in Proposition \ref{generalized:ito:lemma} and the semigroup property satisfied by $(\bar{P}_t)_{t\geq0}$. We conclude that \A{H1} and \A{H2} are satisfied. The proof is now complete.

\subsection{Proof of Theorem \ref{existence:bivariate:density:local:time}}$\,$\\\label{second:part:appendix}

As already mentioned, we prove the result for $b\equiv0$. In order to include a drift, we refer to Remark \ref{bounded:measurable:drift}. We examine the $n$-th term of the series \eqref{semigroup:series:local:time} and prove an important smoothing property of the kernel. More precisely, let $x = x_{n+1}$, $\ell = \ell_{n+1}$ and $s_0 = T$, we claim the following key inequality
\begin{align}
 | \int_{(\rr\times \rr_+)^{n}}  & \hat p_{s_n}(x_0,\ell_0, x_1, \ell_1) \left\{\prod_{i=1}^{n} \hat{\theta}_{s_{n-i}-s_{n-i+1}}(x_{i}, \ell_{i}, x_{i+1}, \ell_{i+1}) u(x_{0},\ell_{0}, x_{i+1},\ell_{i+1}, dx_{i},d\ell_{i})\right\} | \label{recursive:bound:local:time}\\
&  \leq  \prod_{i=1}^{n} C (s_{i-1}-s_{i})^{-1+\frac\eta 2}\times  \left\{ \frac{1}{T^{\frac12}} H_0(c T,  |x| + |x_0| + \ell - \ell_0) \I_\seq{\ell_0< \ell }\right. \nonumber \\
& \qquad \qquad \left. + \left\{\frac{|x|^{\beta}}{T^{\frac{\beta}{2}} } \wedge \frac{|x_0|^{\beta}}{T^{\frac{\beta}{2}} }  \wedge 1  \right\} H_0(c T, x - x_0) \I_\seq{xx_0 \geq 0} \I_\seq{\ell=\ell_0} \right\}\nonumber 
\end{align}

\noindent for any $\beta \in [0,1]$. From the previous bound and Lemma \ref{beta:type:integral} we deduce that 
\begin{align*}
& |p^n_T(x_0,\ell_0,x,\ell)|\\
& \leq |\int_{\Delta_n(T)}  d\s_n  \int_{(\rr\times \rr_+)^{n}}  \hat p_{s_n}(x_0,\ell_0, x_1, \ell_1) \left\{\prod_{i=1}^{n} \hat{\theta}_{s_{n-i}-s_{n-i+1}}(x_{i}, \ell_{i}, x_{i+1}, \ell_{i+1}) u(x_{0},\ell_{0}, x_{i+1},\ell_{i+1}, dx_{i},d\ell_{i})\right\} | \nonumber \\
&  \leq  \int_{\Delta_n(T)} d\s_n \prod_{i=1}^{n} C (s_{i-1}-s_{i})^{-1+\frac\eta 2} \times \left\{ \frac{1}{T^{\frac12}} H_0(c T,  |x| + |x_0| + \ell - \ell_0) \I_\seq{\ell_0< \ell } \right.\\
&\left.\qquad \qquad  \qquad  \qquad   \qquad    +  \left\{\frac{|x|^{\beta}}{T^{\frac{\beta}{2}} } \wedge \frac{|x_0|^{\beta}}{T^{\frac{\beta}{2}} }  \wedge 1 \right\} H_0(c T, x - x_0) \I_\seq{xx_0 \geq 0} \I_\seq{\ell=\ell_0}  \right\} \\
& = \frac{(C T^{\eta/2} \Gamma(\eta/2))^N}{\Gamma(1+N \eta/2)}\left\{ \frac{1}{T^{\frac12}} H_0(c T,  |x| + |x_0| + \ell - \ell_0)  \I_\seq{\ell_0< \ell } +  \left\{\frac{|x|^{\beta}}{T^{\frac{\beta}{2}} } \wedge \frac{|x_0|^{\beta}}{T^{\frac{\beta}{2}} }  \wedge 1  \right\} H_0(c T, x - x_0)  \I_\seq{xx_0 \geq 0} \I_\seq{\ell=\ell_0}  \right\} 
\end{align*}

\noindent with the convention $s_0=T$.
Hence, from Fubini's theorem, the semigroup series obtained from Corollary \ref{corolllary:semigroup} admits the following integral representation
$$
P_T g(x_0,\ell_0)  = \int_{\rr \times \rr_+} g(x,\ell) \left(\sum_{n\geq 0} p^{n}_T(x_0,\ell_0, x,\ell) \right) \, \nu(x_0,\ell_0, dx, d\ell)
$$
\noindent where $p^{n}_T(x_0,\ell_0, x, \ell)$ is given by \eqref{pn:local:time}.
Moreover, from the above inequality, for any $(x_0,\ell_0)$, $(x,\ell) \in \rr \times \rr_+$, one gets the following Gaussian upper bounds
\begin{align*}
&\Big| \sum_{n\geq0} p^{n}_T(x_0,\ell_0,x,\ell) \Big|  \\
& \leq C_T \left\{ \frac{1}{T^{\frac12}} H_0(c T,  |x| + |x_0| + \ell - \ell_0)  \I_\seq{\ell_0< \ell } +  \left\{\frac{|x|^{\beta}}{T^{\frac{\beta}{2}} } \wedge \frac{|x_0|^{\beta}}{T^{\frac{\beta}{2}} }  \wedge 1  \right\} H_0(c T, x - x_0)  \I_\seq{xx_0 \geq 0} \I_\seq{\ell=\ell_0}  \right\}
\end{align*}
\noindent where $C_T :=\sum_{N\geq1} (CT^{\eta/2}\Gamma(\eta/2))^N/\Gamma(1+N \eta/2) < \infty$, for some constants $C, c>1$. 

The proof will be complete once we prove \eqref{recursive:bound:local:time}. We proceed by induction and show that for $j = 1,\dots, n$, the following estimate holds
\begin{align}
| \int_{(\rr\times \rr_+)^{j}} & \hat p_{s_{n}}(x_0,\ell_0, x_1, \ell_1) \left\{\prod_{i=1}^{j} \hat{\theta}_{s_{n-i}-s_{n-i+1}}(x_{i}, \ell_{i}, x_{i+1}, \ell_{i+1}) u(x_{0},\ell_{0}, x_{i+1},\ell_{i+1}, dx_{i},d\ell_{i})\right\} | \nonumber \\
&  \leq  \prod_{i=1}^{j} C (s_{n-i}-s_{n-i+1})^{-1+\frac\eta 2} \left\{( \frac{1}{s_{n-j}^{\frac12}} H_0(c s_{n-j},  |x_{j+1}| + |x_0| + \ell_{j+1} - \ell_0) \I_\seq{\ell_0< \ell_{j+1} } \right. \nonumber \\
& \quad  \left.+  \left\{\frac{|x_{j+1}|^{\beta}}{s_{n-j}^{\frac{\beta}{2}} } \wedge \frac{|x_0|^{\beta}}{s_{n-j}^{\frac{\beta}{2}} }  \wedge 1  \right\} H_0(c s_{n-j}, x_{j+1} - x_0) \I_\seq{x_{j+1}x_0 \geq 0} \I_\seq{\ell_{j+1}=\ell_0} \right\} .\label{induction:hypothesis:local:time:prev}
\end{align}

We start by proving a one step estimate, namely we compute an upper bound for
\begin{align}
\label{first:step:local:time}
\int_{\rr \times \rr_+} & \hat{p}_{s_n}(x_0,\ell_0, x_1, \ell_1)  \hat{\theta}_{s_{n-1}-s_{n}}(x_1,\ell_1, x_{2}, \ell_{2}) u(x_{0},\ell_{0}, x_2,\ell_2,dx_1, d\ell_1).
\end{align}

This term can be decomposed as follows $A_1 \I_\seq{\ell_2=\ell_0} \I_\seq{x_2 x_0 \geq 0} + (A_2+A_3+A_4)\I_\seq{\ell_0 < \ell_2}$. More precisely, on the set $\seq{\ell_2=\ell_0} \cap \seq{x_2 x_0 \geq 0}$, equation \eqref{first:step:local:time} is equal to
$$
\int_{\rr } \hat{p}_{s_n}(x_0,\ell_0, x_1, \ell_0)  \hat{\theta}_{s_{n-1}-s_{n}}(x_1,\ell_1, x_{2}, \ell_{1}) \I_\seq{x_0 x_1 \geq 0} dx_1 := A_1.
$$

On the set $\seq{\ell_0 < \ell_2}$, equation \eqref{first:step:local:time} is equal to
\begin{align*}
&\int_{\rr \times \rr_+}  \hat{p}_{s_n}(x_0,\ell_0, x_1, \ell_1)  \hat{\theta}_{s_{n-1}-s_{n}}(x_1,\ell_1, x_{2}, \ell_{2}) \I_\seq{x_1 x_0 \geq0} dx_1 \delta_{\ell_0}(d\ell_1) \\
& + \int_{\rr \times \rr_+}  \hat{p}_{s_n}(x_0,\ell_0, x_1, \ell_1)  \hat{\theta}_{s_{n-1}-s_{n}}(x_1,\ell_1, x_{2}, \ell_{2}) \I_\seq{x_2 x_1 \geq0} dx_1 \delta_{\ell_2}(d\ell_1) \\
& + \int_{\rr \times \rr_+}  \hat{p}_{s_n}(x_0,\ell_0, x_1, \ell_1)  \hat{\theta}_{s_{n-1}-s_{n}}(x_1,\ell_1, x_{2}, \ell_{2}) \I_\seq{\ell_0 \leq \ell_1 \leq \ell_2} dx_1 d\ell_1 \\
& := A_2 + A_3 + A_4
\end{align*}

From the space-time inequality and Lemma \ref{estimate:kernel:proxy}, for all $\beta\in [0,1]$, one has
\begin{align*}
|A_1| & \leq C \int_{\rr} \left\{\frac{|x_{1}|^{\beta}}{s_n^{\frac\beta 2}}\wedge \frac{|x_{0}|^{\beta}}{s_n^{\frac\beta 2}} \wedge 1 \right\} \I_\seq{x_{1} x_0 \geq 0} H_0(cs_n, x_1-x_0) \left\{\frac{|x_{2}|^{\beta}}{(s_{n-1}-s_n)^{\frac{2  - \eta + \beta}{2}}} \wedge \frac{|x_{1}|^{\beta}}{(s_{n-1}-s_n)^{\frac{2  - \eta + \beta}{2}}} \wedge \frac{1}{(s_{n-1}-s_n)^{1-\frac\eta 2}}   \right\}  \\
& \quad \times H_0(c(s_{n-1}-s_n), x_{2}-x_1) dx_1 \\
& \leq \frac{C}{(s_{n-1}-s_{n})^{1-\frac{\eta}{2}}}  \left\{\frac{|x_{2}|^{\beta}}{s^{\frac{\beta}{2}}_{n-1}} \wedge \frac{|x_{0}|^{\beta}}{s^{\frac{\beta}{2}}_{n-1} } \wedge 1  \right\} H_{0}(cs_{n-1}, x_{2}-x_0)
\end{align*}
\noindent where we separated the two cases $s_{n} \in (0,s_{n-1}/2)$ and $s_{n} \in (s_{n-1}/2,s_{n-1})$ for the last inequality. Indeed, if $s_{n} \in (0,s_{n-1}/2)$, one has $(s_{n-1}-s_{n}) \asymp s_{n-1}$ so that using the inequality 
$$\frac{|x_{2}|^{\beta}}{(s_{n-1}-s_n)^{\frac{2  - \eta + \beta}{2}}} \wedge \frac{|x_{1}|^{\beta}}{(s_{n-1}-s_n)^{\frac{2  - \eta + \beta}{2}}} \wedge \frac{1}{(s_{n-1}-s_n)^{1-\frac\eta 2}}    \leq \frac{C}{(s_{n-1}-s_{n})^{1-\frac{\eta}{2}}} \left\{\frac{|x_{2}|^{\beta}}{s_{n-1}^{\frac{\beta}{2}}} \wedge 1\right\}$$
 and the semigroup property of the Gaussian kernel, one gets 
$$
|A_1| \leq  \frac{C}{(s_{n-1}-s_{n})^{1-\frac{\eta}{2}}}  \left\{\frac{|x_{2}|^{\beta}}{s^{\frac{\beta}{2}}_{n-1}} \wedge 1  \right\} H_{0}(cs_{n-1}, x_{2}-x_0).
$$ 

To obtain the bound with $|x_0|^{\beta}/s^{\frac{\beta}{2}}_{n-1}$, we notice that $|x_1|^{\beta} \leq C(|x_1-x_0|^{\beta}+|x_0|^{\beta})$ and use the following bound 
\begin{align*}
|A_1| & \leq C \frac{|x_{0}|^{\beta}}{s^{\frac{\beta}{2}}_{n-1}} \int_{\rr} \ \frac{1}{s_n^{\frac\beta 2}}  \I_\seq{x_{1} x_0 \geq 0} H_0(cs_n, x_1-x_0) 
 \frac{|x_{1}-x_0|^{\beta}}{(s_{n-1}-s_n)^{1-\frac{ \eta }{2}}}  H_0(c(s_{n-1}-s_n), x_{2}-x_1) dx_1 \\
 & \quad  +C  \frac{|x_0|^{\beta}}{s^{\frac{\beta}{2}}_{n-1}} \int_{\rr} \  \I_\seq{x_{1} x_0 \geq 0} H_0(cs_n, x_1-x_0) 
 \frac{1}{(s_{n-1}-s_n)^{1-\frac{ \eta }{2}}}  H_0(c(s_{n-1}-s_n), x_{2}-x_1) dx_1 \\
 & \leq \frac{C}{(s_{n-1}-s_{n})^{1-\frac{\eta}{2}}}   \frac{|x_{0}|^{\beta}}{s^{\frac{\beta}{2}}_{n-1} }  H_{0}(cs_{n-1}, x_{2}-x_0)
\end{align*}

\noindent where we used the space-time inequality for the last inequality. This proves the desired bound for $s_{n}\in(0,s_{n-1}/2)$ and the second case $s_{n}\in (s_{n-1}/2,s_{n-1})$ follows from similar arguments.

Again from Lemma \ref{estimate:kernel:proxy}, with $\beta = 1$, one has 
\begin{align*}
|A_2| & \leq C \int_{\rr } \left\{\frac{|x_{1}|^{\beta}}{|s_n|^{\frac\beta 2}} \wedge \frac{|x_{0}|^{\beta}}{|s_n|^{\frac\beta 2}} \wedge  1 \right\}  H_0(c s_{n}, x_1-x_0) \I_\seq{x_1 x_0 \geq0} \frac{1}{(s_{n-1}-s_{n})^{\frac{3-\eta}{2}}} H_0(c(s_{n-1}-s_n),|x_{2}|+|x_1|+ \ell_{2}-\ell_0) dx_1 \\
& \leq C \left\{\frac{1}{s^{\frac{\beta}{2}}_{1}(s_{n-1}-s_{n})^{\frac{3-\eta-\beta}{2}}} \wedge \frac{1}{(s_{n-1}-s_{n})^{\frac{3-\eta}{2}}}\right\} H_0(cs_{n-1},|x_{2}|+x_0+\ell_{2}-\ell_0) \\
& \leq \frac{C}{(s_{n-1}-s_n)^{1-\frac{\eta}{2}}} \frac{1}{s^{\frac12}_{n-1}} H_0(c s_{n-1},|x_{2}|+x_0+ \ell_{2}-\ell_0)
\end{align*}

\noindent where we used the inequality $\frac{1}{s^{\frac{\beta}{2}}_{n}(s_{n-1}-s_{n})^{\frac{3-\eta-\beta}{2}}} \wedge \frac{1}{(s_{n-1}-s_{n})^{\frac{3-\eta}{2}}} \leq  \frac{1}{(s_{n-1}-s_{n})^{\frac{3-\eta}{2}}} \leq C\frac{1}{(s_{n-1}-s_{n})^{\frac{1-\eta}{2}}}  \frac{1}{s^{\frac12}_{n-1}}$ for $s_{n} \in (0,s_{n-1}/2)$ and $\frac{1}{s^{\frac{\beta}{2}}_{n}(s_{n-1}-s_{n})^{\frac{3-\eta-\beta}{2}}} \wedge \frac{1}{(s_{n-1}-s_{n})^{\frac{3-\eta}{2}}} \leq  \frac{1}{s^{\frac{\beta}{2}}_{n}(s_{n-1}-s_{n})^{\frac{3-\eta-\beta}{2}}} \leq \frac{C}{(s_{n-1}-s_{n})^{1-\frac{\eta}{2}}}  \frac{1}{s^{\frac12}_{n}} \leq \frac{C}{(s_{n-1}-s_{n})^{1-\frac{\eta}{2}}}  \frac{1}{s^{\frac12}_{n-1}}$ for $s_{n} \in (s_{n-1}/2,s_{n-1})$.  \\

Similarly, from Lemma \ref{estimate:kernel:proxy} with $\beta=1$, one has
\begin{align*}
| A_3 | &  \leq  \int_{\rr} \frac{C}{s^{\frac12}_{n}} H_0(c s_n,|x_1|+|x_0|+\ell_{1}-\ell_0) \frac{1}{(s_{n-1}-s_{n})^{1-\frac{\eta}{2}}} \left\{\frac{|x_{1}|^{\beta}}{(s_{n-1}-s_n)^{\frac\beta 2}} \wedge 1\right\} H_0(c(s_{n-1}-s_n),x_{2}-x_1) \I_\seq{x_{2}x_1 \geq0} dx_{1} \\
& \leq  \frac{C}{(s_{n-1}-s_{n})^{1-\frac{\eta}{2}}}\frac{1}{s^{\frac12}_{n-1}} H_0(c s_{n-1}, |x_{2}|+|x_0| + \ell_{2}-\ell_0)
\end{align*}

\noindent where we separated the computations into the two cases $s_{n}\in (0,s_{n-1}/2)$ and $s_{n}\in (s_{n-1}/2, s_{n-1})$ and followed similar arguments to the previously cases.  
Finally, from Lemma \ref{convolution:local:time}, one has
\begin{align*}
|A_4| & \leq C \int_{\rr \times (\ell_0,\ell_{2}) }  \frac{1}{s^{\frac12}_{n}} H_0(c s_n,|x_1|+|x_0|+\ell_{1}-\ell_0)   \frac{1}{(s_{n-1}-s_{n})^{\frac{3-\eta}{2}}} H_0(c (s_{n-1}-s_{n}), |x_{2}| + |x_{1}|+\ell_{2}-\ell_1)  dx_{1} d\ell_1 \\
& \leq \frac{C}{(s_{n-1}-s_{n})^{1-\frac{\eta}{2}}} \frac{1}{s^{\frac12}_{n-1}}H_0(c s_{n-1}, |x_{2}| + |x_0| +\ell_{2}-\ell_0) .
\end{align*}

Combining all the previous computations, one gets
\begin{align}
|\int_{\rr \times \rr_+} & \hat{p}_{s_n}(x_0,\ell_0, x_1, \ell_1)  \hat{\theta}_{s_{n-1}-s_{n}}(x_1,\ell_1, x_{2}, \ell_{2}) u(x_{0},\ell_{0}, x_2,\ell_2,dx_1, d\ell_1)|  \nonumber \\
&  \leq \frac{C}{(s_{n-1}-s_{n})^{1-\frac\eta 2}}  \frac{1}{s^{\frac12}_{n-1}} H_0(c s_{n-1}, |x_2| + |x_0| + \ell_{2}-\ell_0) \, \I_\seq{\ell_0 < \ell_{2}} \nonumber \\
& \quad  + \frac{C}{(s_{n-1}-s_n)^{1-\frac\eta 2}}  \left\{\frac{|x_{2}|^{\beta}}{s^{\frac{\beta}{2}}_{n-1} } \wedge \frac{|x_{0}|^{\beta}}{s^{\frac{\beta}{2}}_{n-1} } \wedge 1  \right\} H_0(c s_{n-1}, x_{2}-x_0) \, \I_\seq{x_{2} x_0 \geq0} \I_\seq{\ell_{2}=\ell_0} \label{first:iteration}
\end{align}

\noindent for any $\beta\in [0,1]$ and for some positive constant $C$ depending only on the coefficient $\sigma$. 

 Now, we assume that the bound given in \eqref{induction:hypothesis:local:time:prev} is valid at step $j$ and we prove that a similar bound holds at step $j+1$, namely
\begin{align}
| \int_{(\rr\times \rr_+)^{j+1}} & \hat p_{s_n}(x_0,\ell_0, x_1, \ell_1) \left\{\prod_{i=1}^{j+1} \hat{\theta}_{s_{n-i}-s_{n-i+1}}(x_{i}, \ell_{i}, x_{i+1}, \ell_{i+1}) u(x_{0},\ell_{0}, x_{i+1},\ell_{i+1}, dx_{i},d\ell_{i})\right\} | \nonumber \\
&  \leq  \prod_{i=1}^{j+1} C (s_{n-i}-s_{n-i+1})^{-1+\frac\eta 2} \left\{ \frac{1}{s_{n-(j+1)}^{\frac12}} H_0(c s_{n-(j+1)},  |x_{j+2}| + |x_0| + \ell_{j+2} - \ell_0) \I_\seq{\ell_0< \ell_{j+2} }\right. \nonumber \\
& \quad  \left.+  \left\{\frac{|x_{j+2}|^{\beta}}{s_{n-(j+1)}^{\frac{\beta}{2}} } \wedge \frac{|x_0|^{\beta}}{s_{n-(j+1)}^{\frac{\beta}{2}} }  \wedge 1  \right\} H_0(c s_{n-(j+1)}, x_{j+2} - x_0) \I_\seq{x_{j+2}x_0 \geq 0} \I_\seq{\ell_{j+2}=\ell_0} \right\} .\label{induction:hypothesis:local:time:next}
\end{align}
From \eqref{induction:hypothesis:local:time:prev}, the left-hand side of \eqref{induction:hypothesis:local:time:next} is bounded by
\begin{align*}
 &  \prod_{i=1}^{j} C (s_{n-i}-s_{n-i+1})^{-1+\frac\eta 2} \int_{\rr \times \rr_+} \left\{ \frac{1}{s_{n-j}^{\frac12}} H_0(c s_{n-j},  |x_{j+1}| + |x_0| + \ell_{j+1} - \ell_0) \I_\seq{\ell_0< \ell_{j+1} }  +  \left\{\frac{|x_{j+1}|^{\beta}}{s_{n-j}^{\frac{\beta}{2}} } \wedge \frac{|x_0|^{\beta}}{s_{n-j}^{\frac{\beta}{2}} }  \wedge 1  \right\}\right. \\ 
 &  \times H_0(c s_{n-j}, x_{j+1} - x_0) \I_\seq{x_{j+1} x_0 \geq 0} \I_\seq{\ell_{j+1}=\ell_0} \Bigg\} \times \hat{\theta}_{s_{n-(j+1)}-s_{n-j}}(x_{j+1}, \ell_{j+1}, x_{j+2}, \ell_{j+2}) u(x_{0},\ell_{0}, x_{j+2},\ell_{j+2}, dx_{j+1},d\ell_{j+1})
\end{align*}

\noindent which is in turn equal to $\prod_{i=1}^{j+1} C (s_{n-i}-s_{n-i+1})^{-1+\frac\eta 2} \left\{A_1\I_\seq{x_{j} x_0\geq 0}\I_\seq{\ell_{j+2}=\ell_0} + (A_2+ A_3 + A_4) \I_\seq{\ell_0 < \ell_{j+2}} \right\}$ with
\begin{align*}
A_1 & := \int_{\rr} \left\{\frac{|x_{j+1}|^{\beta}}{s_{n-j}^{\frac\beta 2}} \wedge \frac{|x_{0}|^{\beta}}{s_{n-j}^{\frac\beta 2}} \wedge 1 \right\} H_0(cs_{n-j}, x_{j+1}-x_0) \I_\seq{x_{j+1}x_0 \geq 0} \frac{1}{(s_{n-(j+1)}-s_{n-j})^{1-\frac\eta 2}}  \\
& \quad \times \left\{  \frac{|x_{j+2}|^{\beta}}{(s_{n-(j+1)}-s_{n-j})^{\frac{\beta}{2}}}  \wedge \frac{|x_{j+1}|^{\beta}}{(s_{n-(j+1)}-s_{n-j})^{\frac{\beta}{2}}}  \wedge  1\right\} H_0(c(s_{n-(j+1)}-s_{n-j}), x_{j+2}-x_{j+1}) \I_\seq{x_{j+2} x_{j+1} \geq0} dx_{j+1}  \\
& \leq \frac{C}{(s_{n-(j+1)}-s_{n-j})^{1-\frac{\eta}{2}}}  \left\{\frac{|x_{j+2}|^{\beta}}{s^{\frac{\beta}{2}}_{n-(j+1)}} \wedge  \frac{|x_{0}|^{\beta}}{s^{\frac{\beta}{2}}_{n-(j+1)}} \wedge 1 \right\} H_{0}(cs_{n-(j+1)}, x_{j+2}-x_0) \I_\seq{x_{j+2}x_0 \geq0}
\end{align*}
\noindent where the last inequality follows from considering the two cases $s_{n-j}\in (0,s_{n-(j+1)}/2)$ and $s_{n-j} \in (s_{n-(j+1)}/2,s_{n-(j+1)})$ and following similar arguments to the one used in the first step. 

 Again from Lemma \ref{estimate:kernel:proxy}, for any $\beta\in [0,1]$, one has 
\begin{align*}
A_2 & :=  \int_{\rr } \left\{\frac{|x_{j+1}|^{\beta}}{s_{n-j}^{\frac\beta 2}} \wedge \frac{|x_{0}|^{\beta}}{s_{n-j}^{\frac\beta 2}} \wedge 1 \right\} H_0(cs_{n-j}, x_{j+1}-x_0) \I_\seq{x_{j+1}x_0 \geq 0} \frac{1}{(s_{n-(j+1)}-s_{n-j})^{\frac{3-\eta}{2}}} \\
& \quad \times H_0(c(s_{n-(j+1)}-s_{n-j}),|x_{j+2}|+|x_{j+1}|+ \ell_{j+2}-\ell_0) dx_{j+1}  \\
& \leq   \left\{\frac{1}{s^{\frac{\beta}{2}}_{n-j}(s_{n-(j+1)}-s_{n-j})^{\frac{3-\eta-\beta}{2}}} \wedge \frac{1}{(s_{n-(j+1)}-s_{n-j})^{\frac{3-\eta}{2}}}\right\}  H_0(cs_{n-(j+1)},|x_{j+2}|+ |x_0| +\ell_{j+2}-\ell_0) \\
& \leq \frac{C}{(s_{n-(j+1)}-s_{n-j})^{1-\frac\eta 2}}  \frac{1}{s^{\frac12}_{n-(j+1)}} H_0(c s_{n-(j+1)},|x_{j+2}|+|x_0|+ \ell_{j+2}-\ell_0) 
\end{align*}

\noindent where we again separated the computations into the two cases $s_{n-j}\in (0,s_{n-(j+1)}/2)$ and $s_{n-j}\in (s_{n-(j+1)}/2, s_{n-(j+1)})$. Similarly, from Lemma \ref{estimate:kernel:proxy} with $\beta=1$, one has
\begin{align*}
A_3&  := C \int_{\rr} \frac{1}{s^{\frac12}_{n-j}} H_0(c s_{n-j},|x_{j+1}|+|x_0|+\ell_{j+1}-\ell_0) \frac{1}{(s_{n-(j+1)}-s_{n-j})^{1-\frac{\eta}{2}}} \left\{\frac{|x_{j}|^{\beta}}{(s_{n-(j+1)}-s_{n-j})^{\frac\beta 2}} \wedge 1\right\} \\
& \quad \times H_0(c(s_{n-(j+1)}-s_{n-j}),x_{j+2}-x_{j+1}) \I_\seq{x_{j+1}x_{j+2} \geq 0} dx_{j+1} \\
& \leq  \frac{C}{(s_{n-(j+1)}-s_{n-j})^{1-\frac{\eta}{2}}}\frac{1}{s^{\frac12}_{n-(j+1)}} H_0(c s_{n-(j+1)}, |x_{j+2}|+|x_0| + \ell_{j+2}-\ell_0)
\end{align*}

\noindent where we followed similar arguments as done for the first step. Finally, from Lemma \ref{convolution:local:time}, one has
\begin{align*}
A_4 & := C \int_{\rr \times (\ell_0,\ell_{j}) }  \frac{1}{s^{\frac12}_{j+1}} H_0(c s_{n-j},|x_{j+1}|+|x_0|+\ell_{j+1}-\ell_0)   \frac{1}{(s_{n-(j+1)}-s_{n-j})^{\frac{3-\eta}{2}}}\\
&\quad  \times  H_0(c (s_{n-(j+1)}-s_{n-j}), |x_{j+2}| + |x_{j+1}|+\ell_{j+2}-\ell_{j+1})  dx_{j+1} d\ell_{j+1} \\
& \leq \frac{C}{(s_{n-(j+1)}-s_{n-j})^{1-\frac{\eta}{2}}} \frac{1}{s^{\frac12}_{n-(j+1)}}H_0(c s_{n-(j+1)}, |x_{j+2}| + |x_0| +\ell_{j+2}-\ell_0) .
\end{align*}
Hence \eqref{induction:hypothesis:local:time:next} is valid and therefore by induction, the estimate 	\eqref{induction:hypothesis:local:time:prev} holds for $j=1,\dots, n$. Now, the Gaussian bound \eqref{recursive:bound:local:time} follows from \eqref{induction:hypothesis:local:time:prev} by taking $j=n$ and applying the change of variable $k = n-i$.

\subsection{Proof of Theorem \ref{bivariate:density:maximum}}\rule[-10pt]{0pt}{10pt}\\

From the expression of $\bar{P}_t f$, we remark that $(x_0,m_0) \mapsto \bar{P}_t f(x_0,m_0) \in \mathcal{C}^{2,1}_b(\mathcal{J})$ if $f\in \mathcal{C}^{\infty}_b(\mathcal{J})$ and satisfies the condition $\partial_2 \bar{P_t}f(m_0,m_0)= \lim_{x_0 \uparrow m_0} \partial_2 \bar{P}_tf(x_0,m_0)=0$, $m_0 \in \rr$. Moreover, simple computations (that we omit here) yield $\partial_t \bar{P}_t f(x_0,m_0) = \bar{\mathcal{L}}\bar{P}_t f(x_0,m_0)$, $t>0$, $(x_0,m_0) \in \mathcal{J}$. Then, for $(x_0,m_0) \in \mathcal{J}$, we write
\begin{align*}
(\mathcal{L} - \bar{\mathcal{L}}) \bar{P}_t f(x_0,m_0) & =  \frac12(a(x_0,m_0)-a(x_1,m_1)) \partial^2_1 \bar{P}_t f(x_0,m_0) + b(x_0,m_0) \partial_1 \bar{P}_t f(x_0,m_0)\\
										& = \int_{\rr^2} f(x, m) \bar \theta^{(x_1,m_1)}_t(x_0, m_0, x, m) \nu(x_0,m_0, dx, dm)
\end{align*}

\noindent where $\nu(x_0,m_0, dx, dm)  =   \I_\seq{x\leq m} \I_\seq{m_0 < m} dx dm + \I_\seq{x\leq m_0}  dx\delta_{m_0}(dm)$ and 
\begin{align*}
&\bar{\theta}^{(x_1,m_1)}_t(x_0, m_0, x, m)  \\
& :=
\begin{cases}
     \frac12 (a(x_0,m_0) - a(x_1,m_1))(-2H_3)(\bar{a}t,2m-x-x_0) + b(x_0,m_0) (-2H_2)(\bar{a}t,2m-x-x_0),  & x \leq m,\,  m_0 < m,\\
      \frac12 (a(x_0,m_0) - a(x_1,m_1))( H_2(\bar{a}t, x-x_0) - H_2(\bar{a}t,2m_0-x-x_0) ) & \, x< m_0, \, m=m_0\\
         + b(x_0,m_0)  ( H_1(\bar{a}t,2m_0-x-x_0) - H_1(\bar{a}t, x-x_0) ).
\end{cases}
\end{align*}
\smallskip
Hence, we see that \A{H1} (i), (ii), (iii) hold. We now verify \A{H1} (iv) for $\bar{\zeta}=-1$. We proceed as in subsection \ref{first:part:appendix}. From \A{UE}, there exists positive constants $C,\, c>0$ such that for any $t>0$ and $(x_0,m_0) \in \mathcal{J}$, one has
$$
|\bar{p}_t(x_0,m_0, x, m)| \leq C  \left\{ (-H_1)(c t, 2m-x-x_0)\I_\seq{ x \leq m} \I_\seq{ m < m_0} + H_0(c t, x-  x_0)  \I_\seq{m = m_0} \right\}
$$

\noindent and the right-hand side of the above inequality is $\nu(x_0,m_0, dx, dm)$ integrable. In the same spirit, from \A{UE}, \A{HR} and the space-time inequality, we bound $\bar{\theta}^{(x_1,m_1)}_t(x_0, m_0, x, m)$ as follows
\begin{align*}
|\bar{\theta}^{(x_1,m_1)}_t(x_0, m_0, x, m)|  
& \leq  \left\{
    \begin{array}{ll}
       C\left(\frac{1}{t^{\frac32}}+ \frac{|b|_\infty}{t}\right) H_0(ct, 2m-x-x_0), \quad x \leq m,\,  m_0 < m,\\
      C \left(\frac{1}{t}+ \frac{|b|_\infty}{t^{\frac12}}\right) H_0(c t, x-x_0), \quad x< m_0, \, m=m_0.
    \end{array}
\right.
\end{align*}
\smallskip

Integrating with respect to $\nu(x_0,m_0, dx, dm)$ yields
\begin{align*}
\int_{\rr^2} |\bar{\theta}^{(x_1,m_1)}_t(x_0, m_0, x, m)| \nu(x_0,m_0, dx, dm) & \leq C\left(\frac{1}{t^{\frac32}}+ \frac{|b|_\infty}{t}\right) \int_{m_0}^{\infty} \int_{-\infty}^{m} H_0(ct, 2m-x-x_0) dx dm \\
& \quad  + C \left(\frac{1}{t}+ \frac{|b|_\infty}{t^{\frac12}}\right) \int_{-\infty}^{m_0} H_0(c t, x-x_0) dx \\
& \leq C\left(\frac{1}{t^{\frac32}}+ \frac{|b|_\infty}{t}\right) \int^{\infty}_{m_0} (m-m_0) H_0(ct, m-m_0) dm  + \frac{C(1+|b|_{\infty}t^{\frac12})}{t} \\
& \leq  \frac{C(1+|b|_{\infty}t^{\frac12})}{t}.
\end{align*}

Similarly, when we let the freezing point be the end point of the transition density, that is $(x_1,m_1)=(x, m)$, noting that $m-m_0 \leq 2m-x-x_0$, for $m_0\leq m$, $x_0\leq x$, from \A{HR}, \A{UE} and the space-time inequality we get
\begin{align*}
|\bar{\theta}^{(x,m)}_t(x_0, m_0, x, m)| 
& \leq  \left\{
    \begin{array}{ll}
       C\left(\frac{1}{t^{\frac{3-\eta}{2}}}+ \frac{|b|_\infty}{t}\right) H_0(ct, 2m-x-x_0), \quad x \leq m,\,  m_0 < m,\\
      C \left(\frac{1}{t^{1-\frac\eta 2}}+ \frac{|b|_\infty}{t^{\frac12}}\right) H_0(c t, x-x_0), \quad x< m_0, \, m=m_0.
    \end{array}
\right.
\end{align*}
\smallskip
The above estimate in turn implies
\begin{align*}
\int_{\rr^2} |\bar{\theta}^{(x,m)}_t(x_0, m_0, x, m)| \nu(x_0,m_0, dx,dm)&  \leq C \left(\frac{1}{t^{\frac{3-\eta}{2}}}+ \frac{|b|_\infty}{t}\right) \int_{m_0}^{\infty} \int_{-\infty}^{m} H_0(ct, 2m-x-x_0) dx dm  \\
& \quad + C \left(\frac{1}{t^{1-\frac\eta 2}}+ \frac{|b|_\infty}{t^{\frac12}}\right) \int_{-\infty}^{m_0} H_0(c t, x-x_0) dx \\
& \leq  C \left(\frac{1}{t^{\frac{3-\eta}{2}}}+ \frac{|b|_\infty}{t}\right) \int_{m_0}^{\infty} (m-m_0) H_0(ct, m-m_0) dm +  \frac{C_t}{t^{1- \frac{\eta}{2}}}\\
& \leq  \frac{C_t}{t^{1- \frac{\eta}{2}}}
\end{align*}
\smallskip
\noindent where in the second last inequality, we have applied integration by parts with respect to $x$ and taken $C_t= C(1+|b|_\infty t^{\frac{1-\eta}{2}})$. Hence, we conclude that \A{H1} (iv) is satisfied with $\bar \zeta=-1$ and $\zeta=-1+\eta/2$. Assumption \A{H1} (v), \A{H1} (vi), \A{H2} are obtained following the same arguments as in the case of the diffusion process and its running local time, therefore details are omitted.

\smallskip

\subsection{Proof of Theorem \ref{existence:bivariate:density:max}}\rule[-10pt]{0pt}{10pt}\\

In the model of the SDE with its running local time, we examine the $n$-th term of the series \eqref{semigroup:series:max} and we take the convention $x = x_{n+1}$, $m= m_{n+1}$ and $s_0=T$. We prove the following key inequality
\begin{align}
| \int_{(\rr^2)^{n}} & \hat p_{s_n}(x_0, m_0, x_1, m_1)
\left\{\prod_{i=1}^{n} \hat{\theta}_{s_{n-i}-s_{n-i+1}}(x_{i}, m_{i}, x_{i+1}, m_{i+1}) u(x_{0}, m_{0}, x_{i+1}, m_{i+1}, dx_{i},dm_{i})\right\} | \nonumber \\
&  \leq  \prod_{k=1}^{n} C (s_{k-1}-s_{k})^{-1+\frac\eta 2} \left\{ \frac{1}{T^{\frac12}} H_0(c T,  2m - x - x_0)  \I_\seq{x < m} \I_\seq{\ell_0< m } \right.  \nonumber\\
& \left. \quad +  \left\{\frac{|m-x|^{\beta}}{T^{\frac{\beta}{2}} } \wedge \frac{|m-x_0|^{\beta}}{T^{\frac{\beta}{2}} }  \wedge 1  \right\} H_0(c T, x - x_0) \I_\seq{x < m_0} \I_\seq{m= m_0} \right\} \label{recursive:bound:max}
\end{align}

\noindent for any $\beta \in [0,1]$. From the previous bound, we deduce that 
\begin{align*}
& |\int_{\Delta_n(T)}  d\s_n \int_{(\rr^2)^{n}}  \hat p_{s_n}(x_0, m_0, x_1, m_1) \times \left\{\prod_{i=1}^{n} \hat{\theta}_{s_{n-i}-s_{n-i+1}}(x_{i}, m_{i}, x_{i+1}, m_{i+1}) u(x_{0}, m_{0}, x_{i+1}, m_{i+1}, dx_{i},dm_{i})\right\} | \nonumber \\
&  \leq  \int_{\Delta_n(T)} d\s_n \prod_{k=1}^{n} C (s_{k-1}-s_{k})^{-1+\frac\eta 2} \left\{ \frac{1}{T^{\frac12}} H_0(c T,  2m - x - x_0)  \I_\seq{x < m} \I_\seq{m_0< m } \right.  \nonumber\\
& \left. \quad +  \left\{\frac{|m-x|^{\beta}}{T^{\frac{\beta}{2}} } \wedge \frac{|m-x_0|^{\beta}}{T^{\frac{\beta}{2}} }  \wedge 1  \right\} H_0(c T, x - x_0) \I_\seq{x < m_0} \I_\seq{m= m_0} \right\} \\
& = C(N,T)\left\{ \frac{1}{T^{\frac12}} H_0(c T,  2m - x - x_0)  \I_\seq{x < m} \I_\seq{m_0< m } +  \left\{\frac{|m-x|^{\beta}}{T^{\frac{\beta}{2}} } \wedge \frac{|m-x_0|^{\beta}}{T^{\frac{\beta}{2}} }  \wedge 1  \right\} H_0(c T, x - x_0) \I_\seq{x < m_0} \I_\seq{m= m_0} \right\}
\end{align*}

\noindent where, from Lemma \ref{beta:type:integral}, $C(N,T) := \frac{( CT^{\eta/2}\Gamma(\eta/2))^N}{\Gamma(1+N \eta/2)}$.
Hence, from Fubini's theorem, the semigroup series obtained from Corollary \ref{corolllary:semigroup} admits the following integral representation
$$
P_T g(x_0, m_0)  = \int_{\rr^2} g(x, m) \left(\sum_{n\geq 0} p^{n}_T(x_0, m_0, x, m) \right) \, \nu(x_0, m_0, dx, dm)
$$

\noindent where $p^{n}_T(x_0, m_0, x, m)$ is given by \eqref{pn:max}. Moreover, from the above inequality, for any $(x_0,m_0), (x,m) \in \mathcal{J}^2$, one gets the following Gaussian upper bounds
\begin{align}
|p_{T}(x_0,m_0, x,m)|&  : = |\sum_{n\geq 0} p^{n}_T(x_0, m_0, x, m)   | \nonumber \\
& \leq  C_T \left\{\frac{1}{\sqrt{T}}H_0(cT, 2m-x-x_0)  \I_\seq{ x \leq m} \I_\seq{m_0 < m}  + H_0(cT, x-x_0) \I_\seq{x< m_0} \I_\seq{m=m_0}\right\} \label{gaussian:bound:maximum}
\end{align}

\noindent where $C_T :=\sum_{N\geq1} (CT^{\eta/2}\Gamma(\eta/2))^N/\Gamma(1+N \eta/2) < \infty$, for some constants $C, c>1$. Hence it remains to prove \eqref{recursive:bound:max}. Since its proof is similar to the proof of \eqref{recursive:bound:local:time} in the case of local time, we briefly present the guidelines and omit technical details. First we note that from Lemma \ref{estimate:kernel:proxy} and the space-time inequality, the following estimates hold
\begin{align*}
|\hat {p}_t(x_0, m_0, x, m)|  & \leq  
\begin{cases}
       \frac{C}{t^{\frac{1}{2}}}H_0(ct, 2m-x-x_0), 					& \quad  x \leq m,\,  m_0 < m,\\
      C \left\{\frac{|m-x_0|^\beta}{t^{\frac{\beta}{2}}} \wedge \frac{|m-x|^\beta}{t^{\frac{\beta}{2}}} \wedge 1\right\} H_0(c t, x-x_0), & \quad x< m_0, \, m=m_0.\\
\end{cases}
\end{align*}

\noindent  and similarly 
\begin{align*}
|\hat {\theta}_t(x_0, m_0, x, m)|  & \leq  
\begin{cases}
       \frac{C}{t^{\frac{3-\eta}{2}}}H_0(ct, 2m-x-x_0), 					& \quad  x \leq m,\,  m_0 < m,\\
       \frac{C}{t^{1-\frac{\eta}{2}}} \left\{\frac{|m-x_0|^\beta}{t^{\frac{\beta}{2}}} \wedge \frac{|m-x|^\beta}{t^{\frac{\beta}{2}}} \wedge 1 \right\}H_0(c t, x-x_0), & \quad x< m_0, \, m=m_0\\
\end{cases}
\end{align*}
where $\beta \in [0,1]$ can be freely chosen. 
We proceed in a similar fashion to the case of the local time and first compute an upper bound for
\begin{gather}
\int_{\rr^2} \hat{p}_{s}(x_0, m_0, x', m')\hat{\theta}_{t-s}(x', m', x, m)u(x_0,m_0,x,m,dx',dm' ). \label{base:equation}
\end{gather}
In the current case of the maximum there are also four terms to consider. On the set $\left\{m_{0} < m, x\leq m\right\}$, we note that equation \eqref{base:equation} is equal to 
\begin{align*}	
 & \int_{\rr^2}\hat{p}_{s}(x_0, m_0, x', m')\hat{\theta}_{t-s}(x', m', x, m)\I_{\{x'\vee x_0\vee m_{0}<m'\}}  \I_\seq{x\vee x'\vee m' < m} \I_{\{m'<m\}} dx'dm'\\
 & + \int_{\rr^2}\hat{p}_{s}(x_{0}, m_{0}, x', m')\I_\seq{m_0 = m'}\hat{\theta}_{t-s}(x', m', x, m)\I_{\{x\vee x'\vee m'<m\}} \I_{\{x'<m_{0}\}}\delta_{m_0}(dm')dx'\\
 & + \int_{\rr^2}\hat{p}_{s}(x_{0}, m_{0}, x', m')\I_{\{x'\vee x_0\vee m_{0}<m'\}}\hat{\theta}_{t-s}(x', m', x, m) \I_\seq{m' = m} \I_{\{x'<m\}}\delta_{m}(dm')dx' \\
 & =: (A_1+ A_2 + A_3)(x_0,m_0,x,m)
\end{align*}
and, on the set $\left\{m = m_0 \right\}$, \eqref{base:equation} is equal to
\begin{align*}
 \int_{\rr^2} &\hat{p}_s(x_{0}, m_{0}, x', m')\hat{\theta}_{t-s}(x', m', x, m)\I_{\{x' < m_{0}\}}  \I_\seq{x<m'} \I_\seq{m = m'} \delta_{m_0}(dm')dx' \\
& = \I_\seq{x<m_0} \I_\seq{m_0 = m} \int_{\rr} \hat{p}_s(x_{0}, m_{0}, x', m_0)\hat{\theta}_{t-s}(x', m_0, x, m)dx'\\
& =: A_4 (x, m_0, x, m).
\end{align*}
From Lemma \ref{convolution:local:time}, one directly gets 
$$
|A_1| \leq \frac{1}{(t-s)^{1-\frac{\eta}{2}}}\frac{C}{\sqrt{t}} H_0(ct, 2m-x-x_0)\I_{\{x\vee x_0\vee m_0<m\}}.
$$
For the term $A_2$, we notice that $m_0 = m' < m$ and for $\beta= 1$, we obtain the following bound \begin{align*}
|A_2| & \leq \frac{C}{(t-s)^{\frac{3-\eta}{2}}} \int_\rr \left\{\frac{|m_0-x'|^\beta}{s^{\frac{\beta}{2}}}\wedge \frac{|m_0-x_0|}{s^{\frac{\beta}{2}}} \wedge 1 \right\}H_0(c s, x'-x_0)H_0(c(t-s), 2m-x-x') \I_{\{x\vee x'\vee m_0<m\}} \I_{\{x'<m_{0}\}}dx'\\
& \leq \frac{C}{(t-s)^{1-\frac{\eta}{2}+\frac{1}{2}}} \left\{ \frac{|m'-x' + m-x|^\beta}{s^{\frac{\beta}{2}}} \I_\seq{s\in (\frac{t}{2},t]}+ \I_\seq{s\in (0,\frac{t}{2})}  \right\} \\
& \quad \times \int_{\rr} H_0(c s, x'-x_0)  H_0(c(t-s), 2m-x-x')    \I_{\{x\vee x'\vee m_0<m\}} \I_{\{x'<m_{0}\}}dx'\\
& \leq \frac{C}{(t-s)^{1-\frac{\eta}{2}}}\frac{1}{t^{\frac12}}H_0(ct, 2m-x-x_0)\I_{\{x\vee x_0\vee m_0<m\}}
\end{align*}

\noindent where the last inequality follows from the fact that $m'-x' + m-x \leq 2m-x-x'$ and the space-time inequality in the case $s\in (\frac{t}{2},t)$ and Gaussian convolution together with $(t-s)\asymp t$ otherwise.

For the term $A_3$, we note that $m_0<m$ and $|m-x| \leq |m-x + m_0-x_0| \leq |2m-x-x_0|$. Hence, one has
\begin{align*}
|A_3|& \leq \frac{C}{(t-s)^{1-\frac{\eta}{2}}}    \int_{\rr}     \frac{1}{s^{\frac{1}{2}}}H_0(cs, 2m-x'-x_0) \left\{\frac{|m-x|^\beta}{(t-s)^\frac{\beta}{2}}\wedge \frac{|m-x'|^\beta}{(t-s)^\frac{\beta}{2}} \wedge 1 \right\}  H_0(c(t-s), x-x') \I_{\{x'\vee x_0\vee m_0<m\}} \I_{\{x'<m\}}dx'.
\end{align*}
On the set $s\in (\frac{t}{2}, t)$, by the Gaussian convolution, we obtain
\begin{align*}
|A_3| & \leq  \frac{C}{s^{\frac{1}{2}}}\frac{1}{(t-s)^{1-\frac{\eta}{2}}}H_0(ct, 2m-x-x_0)\I_{\{x\vee x_0\vee m_0<m\}} \leq  \frac{C}{(t-s)^{1-\frac{\eta}{2}}} \frac{1}{t^{\frac12}}H_0(ct, 2m-x-x_0)\I_{\{x\vee x_0\vee m_0<m\}}.
\end{align*}
On the set $s\in (0, \frac{t}{2})$, take $\beta = 1$ and by the space-time inequality, one gets
\begin{align*}
|A_3| & \leq     C\frac{1}{s^{\frac{1}{2}}}\frac{1}{(t-s)^{1-\frac{\eta}{2}}}\frac{|m-x|}{(t-s)^\frac{1}{2}} \int_{\rr} H_0(cs, 2m-x'-x_0)H_0(c(t-s), x-x') \I_{\{x'\vee x_0\vee m_0<m\}} \I_{\{x'<m\}}dx'\\
& \leq     C\frac{1}{s^{\frac{1}{2}}}\frac{1}{(t-s)^{1-\frac{\eta}{2}}}\frac{|2m-x-x_0|}{\sqrt{t}} \int_\rr H_0(cs, 2m-x'-x_0)H_0(c(t-s), x-x') \I_{\{x'\vee x_0\vee m_0<m\}} \I_{\{x'<m\}}dx'.\\
& \leq     C\frac{1}{(t-s)^{1-\frac{\eta}{2}}}\frac{1}{t^{\frac12}} H_0(ct, 2m-x-x_0)\I_{\{x\vee x_0\vee m_0<m\}} .
\end{align*}

\noindent For the term $A_4$, we note that $m_0 = m$ so that from Lemma \ref{estimate:kernel:proxy} one gets
\begin{align*}
|A_4| &\leq \frac{C}{(t-s)^{1-\frac{\eta}{2}}} \int_{\rr^2}  \left\{\frac{|m_0-x_0|^\beta}{s^\frac{\beta}{2}}\wedge \frac{|m_0-x'|^{\beta}}{s^{\frac{\beta}{2}}} \wedge 1\right\}H_0(cs, x'-x_0)  \\
&\quad \times \left\{\frac{|m-x|^{\beta}}{(t-s)^\frac{\beta}{2}} \wedge \frac{|m-x'|^{\beta}}{(t-s)^{\frac{\beta}{2}}} \wedge 1 \right\} H_0(c(t-s), x-x')  \I_{\{x' < m_{0}\}}  \I_\seq{x<m'} \I_\seq{m_0 = m} \delta_{m_0}(dm')dx' \\
&\leq \frac{C}{(t-s)^{1-\frac{\eta}{2}}} \int_{\rr^2}\left\{ \left\{ \frac{|m-x_0|^\beta}{s^\frac{\beta}{2}} \wedge ( \frac{|m-x|^{\beta}+|x-x'|^{\beta}}{s^\frac{\beta}{2}} \times ( \frac{|m-x|^{\beta}}{(t-s)^{\frac{\beta}{2}}} \wedge 1) ) \wedge 1\right\} \I_\seq{s\in (\frac{t}{2},t)} \right.  \\
& \quad \left. + \left\{ \frac{|m-x|^\beta}{(t-s)^\frac{\beta}{2}} \wedge (\frac{|m-x_0|^{\beta}+|x_0-x'|^{\beta}}{(t-s)^\frac{\beta}{2}} \times ( \frac{|m-x_0|^{\beta}}{s^{\frac{\beta}{2}}} \wedge 1) ) \wedge 1\right\} \I_\seq{s\in (0,\frac{t}{2})} \right\} \\
& \quad  \quad \times H_0(cs, x'-x_0) H_0(c(t-s),x-x')\I_{\{x' < m_{0}\}}  \I_\seq{x<m'} \I_\seq{m_0 = m} \delta_{m_0}(dm')dx'  \\
& \leq \frac{C}{(t-s)^{1-\frac{\eta}{2}}} \left\{ \frac{|m-x_0|^{\beta}}{t^{\frac{\beta}{2}}} \wedge \frac{|m-x|^{\beta}}{t^{\frac{\beta}{2}}} \wedge 1 \right\}H_0(ct, x-x_0)\I_\seq{x<m_0} \I_\seq{m_0 = m}.
\end{align*}

Therefore summarising the above, the following Gaussian upper bound holds
\begin{align*}
| \int_{\rr^2} \hat{p}_{s}(x_0, m_0, x', m') & \hat{\theta}_{t-s}(x', m', x, m)u(x_0,m_0, x, m, dx',dm' ) | \\ 
 & \leq  
	\begin{cases}
       \frac{C}{(t-s)^{1-\frac{\eta}{2}}}\frac{1}{t^{\frac12}}H_0(ct, 2m-x-x_0),&  \quad x \leq m,\,  m_0 < m,\\
      \frac{C}{(t-s)^{1-\frac{\eta}{2}}}  \left\{ \frac{|m-x|^{\beta}}{t^{\frac{\beta}{2}}} \wedge \frac{|m-x_0|^\beta}{t^\frac{\beta}{2}} \wedge 1 \right\} H_0(ct, x-x_0), & \quad x< m_0, \, m=m_0.
\end{cases}
\end{align*}

\noindent for any $\beta \in [0,1]$.
By an induction argument which is similar to the case of local time, one gets 
\begin{align*}
& |p^n_{T}(x_0,m_0, x,m)| & \leq  
\begin{cases}
       C^n (\int_{\Delta_{n}(T)} d\s_n \,\, \prod_{k=1}^{n} (s_{k-1}-s_{k})^{-1 + \frac{\eta}{2}})    \times \frac{1}{\sqrt{T}}H_0(cT, 2m-x-x_0),&  \quad x\leq m,\,  m_0 < m,\\
      C^n  (\int_{\Delta_{n}(T)} d\s_n \,\, \prod_{k=1}^{n} (s_{k-1}-s_{k})^{-1 + \frac{\eta}{2}})    \times H_0(cT, x-x_0) , & \quad x< m_0, \, m=m_0.
\end{cases}
\end{align*}

We omit technical details. Hence from Lemma \ref{beta:type:integral} and the asymptotic property of the Gamma function, the Gaussian upper bound \eqref{gaussian:bound:maximum} for the transition density is valid. This concludes the proof.

\vskip 10pt
\subsection{Some useful technical results}
\begin{lem}\label{beta:type:integral} Let $b>-1$ and $a\in [0,1)$. Then for any $t_0>0$,
$$
\int_{\Delta_n(t_0)}\,d\mathbf{t}_n  \,\,t^{b}_n \prod_{j=0}^{n-1} (t_j-t_{j+1})^{-a} = \frac{t_0^{b+n(1-a)} \Gamma^{n}(1-a)\Gamma(1+b)}{\Gamma(1+b+n(1-a))}
$$
\end{lem}
\begin{proof}
Using the change of variables $s=ut$, one has
$$
\int_0^t s^{b} (t-s)^{-a} ds = t^{b+1-a} \int_0^1 u^{b} (1-u)^{-a} du = t^{b+1-a} B(1+b,1-a)
$$

\noindent where $(x,y) \mapsto B(x,y)=\int_0^1 t^{x-1}(1-t)^{y-1} dt$ stands for the standard Beta function. Using this equality repeatedly, we obtain the statement.
\end{proof}
\begin{lem}\label{convolution:local:time}Let $c_1 >0$. For any $(x,x_0) \in \rr^2$, $0\leq \ell_0 \leq \ell_2$ and $0< s < t$, one has
$$
 \int_{\rr \times (\ell_0, \ell)} \frac{1}{(t-s)^{\frac12}} H_0(c_1 (t-s), |x| + |x'| + \ell- \ell' ) \frac{1}{s^{\frac12}} H_0(c_1 s, |x'|+ |x_0|+ \ell'- \ell_0) dx' d\ell' \leq \frac{C}{t^{\frac12}} H_0(c t, |x|+ |x_0| + \ell - \ell_0)
$$

\noindent for some positive constants $C, \, c$ independent of $t$, $x_0$, $\ell_0$ and $\ell$.  Similarly, for any $(x_0,m_0) \in \mathcal{J}$, $m\geq m_0$ and $0< s < t$, one has
$$
 \int_{(m_0, m) \times (-\infty,m')} \frac{1}{(t-s)^{\frac12}} H_0(c_1 (t-s), 2m-x-x') \frac{1}{s^{\frac12}} H_0(c_1 s, 2m'-x'-x_0) dx' dm' \leq \frac{C}{t^{\frac12}} H_0(c t, 2m-x-x_0)
$$

\noindent for some positive constants $C, \, c$ independent of $t$, $x_0$, $m_0$ and $m$. 
\end{lem}
\begin{proof} We will only prove the first bound. The second one follows from similar arguments. For simplicity, we write
$$
C_2:= \int_{\rr \times (\ell_0, \ell)} \frac{1}{(t-s)^{\frac12}} H_0(c (t-s), |x| + |x'| + \ell- \ell' ) \frac{1}{s^{\frac12}} H_0(c s, |x'|+ |x_0|+ \ell'- \ell_0) dx' d\ell'.
$$

Let us assume that $|x|+|x_0| + \ell-\ell_0 \leq t^{\frac12}$. We use the fact that the diagonal estimate is global. For $s\in [\frac t2, t]$, one has $s \asymp t$ so that 
$$
 \frac{1}{s^{\frac12}} H_0(c s, |x'|+ |x_0|+ \ell'- \ell_0) \leq Ct^{-1} \leq C H_0(c t, |x|+ |x_0|+ \ell- \ell_0)
$$   

\noindent which in turn implies:
$$
|C_2| \leq \frac{C}{\sqrt{t}} H_0(c t, |x|+ |x_0| + \ell - \ell_0) \int_{\rr \times (\ell_0, \ell)} \frac{1}{(t-s)^{\frac12}} H_0(c (t-s), |x| + |x'| + \ell- \ell') dx' d\ell' \leq \frac{C}{\sqrt{t}} H_0(c t, |x|+ |x_0| + \ell - \ell_0).
$$
Similarly, for $s\in [0,\frac t2]$, one has
$$
|C_2| \leq \frac{C}{\sqrt{t}} H_0(c t, |x|+ |x_0| + \ell - \ell_0).
$$ 
Hence, the claim follows in the diagonal regime $|x|+|x_0| + \ell-\ell_0 \leq t^{\frac12}$.  We now consider the off-diagonal regime $|x|+|x_0| + \ell-\ell_0 > t^{\frac12}$. We write $\rr \times (\ell_0, \ell)= D_1 \cup D_2$ where 
\begin{align*}
D_1& := \left\{ (x',\ell')\in \rr \times (\ell_0,\ell): |x'| + |x_0| + \ell'-\ell_0 \leq |x|-|x'| + \ell- \ell'\right\},   \\
D_2 & := \left\{ (x',\ell')\in \rr \times (\ell_0,\ell): |x'| + |x_0| + \ell'-\ell_0 > |x|-|x'| + \ell- \ell'\right\}.
\end{align*}
On the set $D_1$, we remark $|x|-|x'| + \ell-\ell' \asymp |x|+|x_0| + \ell-\ell_0$, so that
\begin{align*}
 \frac{1}{(t-s)^{\frac12}} H_0(c (t-s), |x| + |x'| + \ell- \ell' ) & \leq \frac{C}{t}  \frac{(|x|+|x_0| + \ell-\ell_0 )^2}{(t-s)^{\frac12}} H_0(c(t-s),|x|+|x_0| + \ell-\ell_0) \\
 & \leq \frac{C}{t^{\frac12}} H_0(ct, |x|+|x_0| + \ell-\ell_0)
\end{align*}

\noindent where we used the space-time inequality for the last inequality. Hence, one gets
$$
\int_{D_1} \frac{1}{(t-s)^{\frac12}} H_0(c (t-s), |x| + |x'| + \ell- \ell' ) \frac{1}{s^{\frac12}} H_0(c s, |x'|+ |x_0|+ \ell'- \ell_0) dx' d\ell' \leq  \frac{C}{t^{\frac12}} H_0(ct, |x|+|x_0| + \ell-\ell_0).
$$

On the set $D_2$, one has $|x|-|x'| + \ell-\ell' \asymp |x|+|x_0| + \ell-\ell_0$, so that
\begin{align*}
 \frac{1}{s^{\frac12}} H_0(c s, |x'| + |x_0| + \ell'- \ell_0 ) & \leq \frac{C}{t} \frac{(|x|+|x_0| + \ell-\ell_0 )^2}{s^{\frac12}} H_0(cs,|x|+|x_0| + \ell-\ell_0) \\
 & \leq \frac{C}{t^{\frac12}} H_0(ct, |x|+|x_0| + \ell-\ell_0)
\end{align*}

\noindent which in turn yields
$$
\int_{D_2} \frac{1}{(t-s)^{\frac12}} H_0(c (t-s), |x| + |x'| + \ell- \ell' ) \frac{1}{s^{\frac12}} H_0(c s, |x'|+ |x_0|+ \ell'- \ell_0) dx' d\ell' \leq  \frac{C}{t^{\frac12}} H_0(ct, |x|+|x_0| + \ell-\ell_0)
$$

\noindent and proves the claim. 
\end{proof}
\begin{lem}\label{estimate:kernel:proxy}
Let $a>0$. Define $\bar{f}_t(x_0,x):= H_0(a t, x-x_0) - H_0(at, x+x_0)$. For any $\beta\in [0,1]$, there exists $C,\ c>1$, such that for any $(x_0,x)\in \rr^2$ such that $x x_0 \geq 0$ and $r=0,2$, the following estimates hold:
\begin{align*}
|\partial^{r}_{x_0}\bar f_t(x_0,x)| & \leq \frac{C}{t^{\frac{r}{2}}}\left\{\frac{|x|^\beta}{t^\frac{\beta}{2}}\wedge \frac{|x_0|^\beta}{t^\frac{\beta}{2}} \wedge 1 \right\} H_0(c t, x-x_0).
\end{align*}
Let $\bar{p}_t(x_0,x) := H_0(at, x-x_0) - H_0(a  t, 2m_0 - x-x_0)$, for $r=0,1,2$, one has
\begin{align*}
|\partial^{r}_{x_0}\bar p_t(x_0,x)| & \leq \frac{C}{t^{\frac r 2}} \left\{\frac{|m_0-x_0|^\beta}{t^{\frac{\beta}{2}}} \wedge \frac{|m_0-x|^\beta}{t^{\frac{\beta}{2}}} \wedge 1 \right\} H_0(c t, x-x_0), \quad x_0, \, x \leq m_0.
\end{align*}
\end{lem}

\begin{proof}
From the expression of $\bar{f}_t(x_0,x)$, the following estimates for $\bar{f}_t(x_0,x)$ and its derivatives hold
\begin{align*}
|\partial^{r}_{x_0}\bar{f}_t(x_0,x)| & \leq \frac{C}{t^{\frac r 2}} ( H_0(c t, x-x_0) +  H_0(c t, x+x_0) ), \quad r=0,2.
\end{align*}
Furthermore, for $(x_0,x)\in \rr^2$ such that $x x_0 \geq0$, one has $H_0(c t, x+x_0) \leq  H_0(c t, x-x_0)$, since in the exponent
\begin{align*}
(x-x_0)^2 + 4 x_0(x-x_0) + 4 x_0^2 & \geq (x-x_0)^2.
\end{align*}

Hence, we deduce that $|\partial^{r}_{x_0}\bar f_t(x_0,x)| \leq C t^{-\frac{r}{2}} g(c t, x-x_0)$. 
To derive the bounds with the $|x|^\beta$ or $|x_0|^{\beta}$ terms, we first consider the case where $|x|^2\leq  t$, to estimate $\bar f_t(x_0,x) = H_0(a(y,\ell) t, x-x_0) - H_0(a(y,\ell)t, x+x_0)$ one applies the mean value theorem to $g(a(y,\ell)t, x-x_0)$ with respect to the points $x$ and $-x$ to obtain for some $\theta \in [0,1]$,
\begin{align*}
|\bar f_t(x_0, x)| & = |2 x \partial_x H_0(a(y,\ell)t, x_0- \theta x  + (1-\theta)x)|  \leq  C \frac{|x|^\beta }{t^{\frac{\beta}{2}}} H_0(c t, x-x_0)
\end{align*}
where in the second line we have used the space-time inequality and the fact that $|x|^{1-\beta} \leq t^\frac{1-\beta}{2}$. For the case that $|x|^2 \geq  t$, one directly gets
\begin{align*}
 |\bar{f}_t(x_0,x)|\leq C\frac{|x|^\beta}{t^\frac\beta2} H_0(c t, x-x_0) .
\end{align*}
The proof for the second derivatives of $\bar{f}^z_t(x_0,x)$ as well as the estimates with the $|x_0|^{\beta}$ term and the estimates for $\partial^{r}_{x_0}\bar{p}_t(x_0,x)$ follow similar arguments and details are omitted.
\end{proof}

\vskip 30pt
\bibliographystyle{alpha}
\bibliography{bibli}
\end{document}